\documentclass[10pt,reqno]{amsart}

\usepackage{fixltx2e}                       
\usepackage[usenames,dvipsnames]{xcolor}
\usepackage{fancyhdr}
\usepackage{amsmath,amsfonts,amsbsy,amsgen,amscd,mathrsfs,amssymb,amsthm}
\usepackage{mathtools}
\usepackage[font=small,margin=0.25in,labelfont={sc},labelsep={colon}]{caption}
\usepackage{tikz}
\usepackage{enumerate}
\usepackage{graphicx}

\usepackage{bm} 

\usepackage[colorlinks=true]{hyperref}

\newtheorem{thm}{Theorem}[section]
\newtheorem{lem}[thm]{Lemma}

\newtheorem{cor}[thm]{Corollary}

\theoremstyle{definition}

\newtheorem{defn}[thm]{Definition}

\newtheorem{rem}[thm]{Remark}

\numberwithin{equation}{section} 
\numberwithin{figure}{section}
\numberwithin{table}{section}

\newcommand{\E}{\mathbf{E}}

\newcommand{\diam}{\mathop{\mathrm{diam}}}
\newcommand{\mynegspace}{\hspace{-0.12em}}
\newcommand{\vvv}{\rvert\mynegspace\rvert\mynegspace\rvert}
\newcommand{\nnn}[1]{{\vvv #1 \vvv}}

\begin{document}

\title[Chaining, Interpolation, and Convexity II]{Chaining, 
Interpolation, and Convexity II: \\ The contraction principle}
\author{Ramon van Handel}
\address{Fine Hall 208, Princeton University, Princeton, NJ 
08544, USA}
\email{rvan@princeton.edu}

\begin{abstract} The generic chaining method provides a sharp description 
of the suprema of many random processes in terms of the geometry of their 
index sets. The chaining functionals that arise in this theory are however 
notoriously difficult to control in any given situation. In the first 
paper in this series, we introduced a particularly simple method for 
producing the requisite multiscale geometry by means of real 
interpolation. This method is easy to use, but does not always yield sharp 
bounds on chaining functionals. In the present paper, we show that a 
refinement of the interpolation method provides a canonical mechanism for 
controlling chaining functionals. The key innovation is a simple but 
powerful contraction principle that makes it possible to efficiently 
exploit interpolation. We illustrate the utility of this 
approach by developing new dimension-free bounds on the norms of random 
matrices and on chaining functionals in Banach lattices. As another 
application, we give a remarkably short interpolation proof of the 
majorizing measure theorem that entirely avoids the greedy construction 
that lies at the heart of earlier proofs. \end{abstract}



\maketitle

\thispagestyle{empty}

\section{Introduction}

The development of sharp bounds on the suprema of random processes is of 
fundamental importance in diverse areas of pure and applied mathematics. 
Such problems arise routinely, for example, in probability theory, 
functional analysis, convex geometry, mathematical statistics, and 
theoretical computer science. 

It has long been understood that the behavior of suprema of random 
processes is intimately connected with the geometry of their index sets. 
This idea has culminated in a remarkably general theory due to 
M.\ Talagrand that captures the precise connection between the underlying 
probabilistic and geometric structures for many interesting types of 
random processes. For example, the classic result in this theory, known 
(for historical reasons) as the majorizing measure theorem, provides a sharp 
geometric characterization of the suprema of Gaussian processes.

\begin{thm}[\cite{Tal14}]
\label{thm:mm}
Let $(X_x)_{x\in T}$ be a centered Gaussian process and denote by
$d(x,y) = (\mathbf{E}|X_x-X_y|^2)^{1/2}$ the associated natural metric on 
$T$.  Then
$$
        \mathbf{E}\bigg[\sup_{x\in T}X_x\bigg]
        \asymp
        \gamma_2^*(T) :=
        \inf\sup_{x\in T}\sum_{n\ge 0}2^{n/2}d(x,T_n),
$$
where the infimum is taken over all sequences of sets $T_n$
with cardinality $|T_n|<2^{2^n}$.
\end{thm}

The method behind the proof of Theorem \ref{thm:mm} is called the generic 
chaining. It is by no means restricted to the setting of Gaussian 
processes, and full or partial analogues of Theorem \ref{thm:mm} exist in 
various settings. We refer to the monograph \cite{Tal14} for a 
comprehensive treatment of this theory and its applications.

The majorizing measure theorem provides in principle a complete geometric 
understanding (up to universal constants) of the suprema of Gaussian 
processes. The chaining functional $\gamma_2^*(T)$ captures the relevant 
geometric structure: it quantifies how well the index set $T$ can be 
approximated, in a multiscale fashion, by increasingly fine 
discrete nets $T_n$. The apparently definitive nature of Theorem 
\ref{thm:mm} belies the fact that this result is often very difficult to 
use in any concrete situation. The problem is that while Theorem 
\ref{thm:mm} guarantees that there must exist some optimal sequence of 
nets $T_n$ that yields a sharp bound on the supremum of any given Gaussian 
process, the theorem does not explain how to find such nets. In many 
cases, straightforward discretization of the index set (Dudley's 
inequality) gives rise to suboptimal bounds, and it is not clear how such 
bounds can be improved.

Even without going beyond the setting of Gaussian processes, there are 
plenty of challenging problems, for example, in random matrix theory 
\cite{RV08,vH16b,vH17}, that remain unsolved due to the lack of 
understanding of how to control the supremum of some concrete Gaussian 
process; in fact, even in cases where the supremum of a Gaussian process 
can be trivially bounded by probabilistic means, the underlying geometry 
often remains a mystery, cf.\ \cite[p.\ 50]{Tal14}. From this perspective, 
the generic chaining theory remains very far from being well understood. 
It is therefore of considerable interest to develop new mechanisms for the 
control of chaining functionals such as $\gamma_2^*(T)$. The aim of this 
paper is to take a further step in this direction.

The main (nontrivial) technique that has been used to date to control 
chaining functionals is contained in the proof of Theorem \ref{thm:mm}. To 
show that $\gamma_2^*(T)$ is bounded above by the expected supremum of a 
Gaussian process, a sequence of nets $T_n$ is constructed by repeatedly 
partitioning the set $T$ in a greedy fashion, using the functional 
$G(A):=\mathbf{E}[\sup_{x\in A}X_x]$ to quantify the size of each 
partition element. It is necessary to carefully select the partition 
elements at each stage of the construction in order to control future 
iterations, which requires fairly delicate arguments (cf.\ \cite[section 
2.6]{Tal14}). It turns out, however, that the proof does not rely heavily 
on special properties of Gaussian processes: the only property of the 
functional $G(A)$ that is used is that a certain ``growth condition'' is 
satisfied. If one can design another functional $F(A)$ that mimics this 
property of Gaussian processes, then the same proof yields an upper bound 
on $\gamma_2^*(T)$ in terms of $F(T)$.

In principle, this partitioning scheme provides a canonical method for 
bounding chaining functionals such as $\gamma_2^*(T)$: it is always 
possible to choose a functional satisfying the requisite growth condition 
that gives a sharp bound on $\gamma_2^*(T)$. This observation has little 
practical relevance, as this conclusion follows from the fact that the 
chaining functional itself satisfies the growth condition (cf.\ \cite[pp.\ 
38--40]{Tal14}) which does not help to obtain explicit bounds on these 
functionals. Nonetheless, this observation shows that no loss is incurred 
in the partitioning scheme \emph{per se}, so that its application is only 
limited by our ability to design good growth functionals that admit 
explicit bounds. Unfortunately, the latter requires considerable 
ingenuity, and has been carried out successfully in a limited number of 
cases.

In the first paper in this series \cite{vH16a}, the author introduced a 
new method to bound chaining functionals that is inspired by real 
interpolation of Banach spaces. The technique developed in \cite{vH16a} is 
completely elementary and is readily amenable to explicit computations, 
unlike the growth functional method. This approach considerably simplifies 
and clarifies some of the most basic ideas in the generic chaining theory, 
such as the construction of chaining functionals on uniformly convex 
bodies. On the other hand, this basic method is not always guaranteed to 
give sharp bounds on $\gamma_2^*(T)$, as can be seen in simple examples 
(cf.\ \cite[section 3.3]{vH16a}). It is therefore natural to expect that 
the utility of the interpolation method may be restricted to certain 
special situations whose geometry is well captured by this construction.

The main insight of the present paper is that this is not the case: 
interpolation provides a canonical method for bounding 
chaining functionals. The problem with the basic method of \cite{vH16a} 
does not lie with the interpolation method itself, but is rather due to 
the fact that this method was inefficiently exploited in its simplest 
form. What is missing is a simple but apparently fundamental ingredient, a 
contraction principle, that will be developed and systematically exploited 
in this paper. Roughly speaking, the contraction principle states that we 
can control chaining functionals such as $\gamma_2^*(T)$ whenever we have 
suitable control on the entropy numbers of all subsets $A\subseteq T$. A 
precise statement of this principle will be given in section 
\ref{sec:contract} below, and its utility will be illustrated throughout 
the rest of paper.

The combination of the interpolation method and the contraction principle 
provides a foundation for the generic chaining theory that yields 
significantly simpler proofs and is easier to use (at least in this 
author's opinion) than the classical approach through growth functionals. 
This approach will be illustrated in a number of old and new applications. 
For example, we will fully recover the majorizing measure theorem with a 
remarkably short proof that does not involve any greedy partitioning 
scheme. The latter is somewhat surprising in its own right, as a greedy 
construction lies very much at the core of earlier proofs of Theorem 
\ref{thm:mm}.

This paper is organized as follows. In section \ref{sec:defn}, we set up 
the basic definitions and notation that will be used throughout. Section 
\ref{sec:contract} develops the main idea of this paper, the contraction 
principle. This principle is first illustrated by means of some elementary 
examples in section \ref{sec:examples}. In section \ref{sec:geom}, we 
develop a geometric principle that resolves a question posed in 
\cite[Remark 4.4]{vH16a}. We then use this principle to develop new 
results on the behavior of chaining functionals on Banach lattices, as 
well as to recover classical results on uniformly convex bodies. In 
section \ref{sec:mm}, we develop a very simple proof of Theorem 
\ref{thm:mm} using the machinery of this paper. We also show that the 
growth functional machinery that lies at the heart of \cite{Tal14} can be 
fully recovered as a special case of our approach. Finally, in section 
\ref{sec:rmt} we take advantage of the methods of this paper to
develop new dimension-free bounds on the operator norms of structured 
random matrices.

\section{Basic definitions and notation}
\label{sec:defn}

The aim of this section is to set up the basic definitions and notation 
that will be used throughout the paper. We introduce a general setting 
that will be specialized to different problems as needed in the sequel.

Let $(X,d)$ be a metric space. We begin by defining entropy numbers.

\begin{defn}
For every $A\subseteq X$ and $n\ge 0$, define the \emph{entropy number}
$$
	e_n(A) := \inf_{|S|<2^{2^n}}\sup_{x\in A} d(x,S).
$$
(In this definition, the net $S\subseteq X$ is not required to be a subset 
of $A$.)
\end{defn}

Another way to interpret $e_n(A)$ is by noting that $A$ can be covered by 
less than $2^{2^n}$ balls of radius $e_n(A)$. It is useful to recall at 
this stage a classical observation (the duality between covering 
and packing) that will be needed below.

\begin{lem}
\label{lem:packing}
Let $n\ge 0$ and $N=2^{2^n}$. Then for every 
$0<\delta<e_n(A)$, there exist points $x_1,\ldots,x_{N}\in A$ such 
that $d(x_i,x_j)>\delta$ for all $i\ne j$.
\end{lem}

\begin{proof}
Select consecutive points $x_1,x_2,\ldots$ as follows: $x_1\in A$ is 
chosen arbitrarily, and $x_i\in A$ is chosen such that 
$d(x_i,x_j)>\delta$ for all $j<i$. Suppose this construction terminates in 
round $M$, that is, there does not exist $x\in A$ such that 
$d(x,x_j)>\delta$ for all $j\le M$. Then setting $S=\{x_1,\ldots,x_M\}$, 
we have $\sup_{x\in A}d(x,S)\le\delta$. Thus $M\ge N$, as otherwise 
$e_n(A)\le\delta$ which contradicts our assumption.
\end{proof}

We now turn to the definition of chaining functionals. For the purposes of 
the present paper, it will be convenient to use a slightly different 
definition than is stated in Theorem \ref{thm:mm} that uses partitions 
rather than nets. We also formulate a more general class of chaining 
functionals that are useful in the more general generic chaining theory 
(beyond the setting of Gaussian processes), cf.\ \cite{Tal14}.

\begin{defn}
Let $T\subseteq X$. An \emph{admissible sequence} of $T$ is an increasing 
sequence $(\mathcal{A}_n)$ of partitions of $T$ such that 
$|\mathcal{A}_n|<2^{2^n}$ for all $n\ge 0$. For every $x\in T$, we denote 
by $A_n(x)$ the unique element of $\mathcal{A}_n$ that contains $x$.
\end{defn}

\begin{defn}
Let $T\subseteq X$. For $\alpha>0$ and $p\ge 1$, define the 
\emph{chaining functional}
$$
	\gamma_{\alpha,p}(T) :=
	\Bigg[\inf \sup_{x\in T}\sum_{n\ge 0}(2^{n/\alpha}
	\diam(A_n(x)))^p\Bigg]^{1/p},
$$
where the infimum is taken over all admissible sequences of $T$. The most 
important case $p=1$ is denoted as
$\gamma_\alpha(T) := \gamma_{\alpha,1}(T)$.
\end{defn}

It is an easy fact that the chaining functional $\gamma_2^*(T)$ that 
appears in Theorem \ref{thm:mm} satisfies $\gamma_2^*(T)\le\gamma_2(T)$: 
given an admissible sequence $(\mathcal{A}_n)$ of $T$, we may simply 
select a net $T_n$ by choosing one point arbitrarily in every element of 
the partition $\mathcal{A}_n$. As our interest in this paper is to obtain 
upper bounds on $\gamma_2(T)$, these trivially give upper bounds on 
$\gamma_2^*(T)$ as well. It is not difficult to show that these quantities 
are actually always of the same order, cf.\ \cite[section 2.3]{Tal14}. 
This will also follow as a trivial application of the main result of 
this paper, see section \ref{sec:adnet} below.

Let us emphasize that the definitions of $e_n(A)$ and 
$\gamma_{\alpha,p}(T)$ depend on the metric of the underlying metric space 
$(X,d)$. In some situations, we will be working with multiple metrics; in 
this case, the metric $d$ that is used to define the above quantities will 
be denoted explicitly by writing $e_n(A,d)$, $\gamma_{\alpha,p}(T,d)$,
and $\diam(A,d)$.

Throughout this paper, we will write $a\lesssim b$ if $a\le Cb$ for a 
universal constant $C$, and we write $a\asymp b$ if $a\lesssim b$ and 
$b\lesssim a$. In cases where the universal constant depends on some 
parameter of the problem, this will be indicated explicitly.

\section{The contraction principle}
\label{sec:contract}

\subsection{Statement of the contraction principle}

At the heart of this paper lies a simple but apparently fundamental 
principle that will be developed in this section. The basic idea is that 
we can control the chaining functionals $\gamma_{\alpha,p}(T)$ whenever we 
have suitable control on the entropy numbers $e_n(A)$ of all subsets 
$A\subseteq T$.

\begin{thm}[Contraction principle]
\label{thm:contr}
Let $s_n(x)\ge 0$ and $a\ge 0$ be chosen so that
$$
	e_n(A) \le a\diam(A) + \sup_{x\in A}s_n(x)
$$
for every $n\ge 0$ and $A\subseteq T$. Then
$$
	\gamma_{\alpha,p}(T) \lesssim
	a\,\gamma_{\alpha,p}(T) +
	\Bigg[\sup_{x\in T}
	\sum_{n\ge 0}(2^{n/\alpha}s_n(x))^p
	\Bigg]^{1/p},
$$
where the universal constant depends only on $\alpha$.
\end{thm}

Of course, this result is of interest only when $a$ can be chosen 
sufficiently small, in which case it immediately yields an upper bound on 
$\gamma_{\alpha,p}(T)$.

It should be emphasized that the only nontrivial aspect of Theorem 
\ref{thm:contr} is to discover the correct formulation of this principle; 
no difficulties of any kind are encountered in the proof. What may be 
far from obvious at present is that this is in fact a powerful or even 
useful principle. This will become increasingly clear in the following 
sections, where we will see that the interpolation method of \cite{vH16a} 
provides a canonical mechanism for generating the requisite controls 
$s_n(x)$.

As Theorem \ref{thm:contr} lies at the core of this paper, we give 
two slightly different proofs.


\subsection{First proof}

The idea of the proof is that the assumption of Theorem \ref{thm:contr} 
allows us to construct from any admissible sequence a new admissible 
sequence that provides more control on the value of the chaining 
functional.

\begin{proof}[First proof of Theorem \ref{thm:contr}]
As $e_n(A)\le \diam(T)$ for every $A\subseteq T$, we can assume without 
loss of generality that $s_n(x)\le\diam(T)$ for all $n,x$.

Let $(\mathcal{A}_n)$ be an admissible sequence of $T$. For every $n\ge 1$
and partition element $A_n\in\mathcal{A}_n$, we construct sets $A_n^{ij}$ 
as follows. We first partition $A_n$ into $n$ segments
\begin{align*}
	A_n^i &:= \{x\in A_n:2^{-2i/\alpha}\diam(T)<s_n(x)\le 
	2^{-2(i-1)/\alpha}\diam(T)\} \quad(1\le i<n),\\
	A_n^n &:= \{x\in A_n:s_n(x) \le 2^{-2(n-1)/\alpha}\diam(T)\}.
\end{align*}
The point of this step is to ensure that
$$
	\sup_{y\in A_n^i}s_n(y) \le 
	2^{2/\alpha}s_n(x)+2^{-2(n-1)/\alpha}\diam(T)
	\quad\mbox{for all }x\in A_n^i, ~i\le n,
$$
that is, that $s_n(x)$ is nearly constant on $A_n^i$.
Using the assumption of the theorem, we can further partition each set 
$A_n^i$ into less than $2^{2^n}$ pieces $A_n^{ij}$ such that
$$
	\diam(A_n^{ij}) \le 2a\diam(A_n) +
	2^{1+2/\alpha}s_n(x) +
	2^{1-2(n-1)/\alpha}\diam(T)
	\quad\mbox{for all }x\in A_n^{ij}.
$$
Let $\mathcal{C}_{n+3}$ be the partition generated by all sets
$A_k^{ij}$, $k\le n$, $i,j$ thus constructed. Then
$|\mathcal{C}_{n+3}|<\prod_{k=1}^n k(2^{2^k})^2<2^{2^{n+3}}$.
Defining $\mathcal{C}_k=\{T\}$ for $0\le k\le 3$, we obtain 
\begin{align*}
	&\gamma_{\alpha,p}(T) \le
	\Bigg[
	\sup_{x\in T}
	\sum_{n\ge 0}(2^{n/\alpha}\diam(C_n(x)))^p\Bigg]^{1/p} \\
	&\lesssim
	a
	\Bigg[
	\sup_{x\in T}
	\sum_{n\ge 0}(2^{n/\alpha}\diam(A_n(x)))^p\Bigg]^{1/p}
	+
	\Bigg[
	\sup_{x\in T}
	\sum_{n\ge 0}(2^{n/\alpha}s_n(x))^p\Bigg]^{1/p}
	+
	\diam(T)
\end{align*}
where the universal constant depends only on $\alpha$. The last term on 
the right can be absorbed in the first two as
$\diam(T)\le 2e_0(T) \le 2a\diam(A_0(x))+2\sup_{x\in T}s_0(x)$.
As the admissible sequence $(\mathcal{A}_n)$ was arbitrary, the conclusion  
follows readily.
\end{proof}

One way to interpret this proof is as follows. We used the assumption of 
Theorem~\ref{thm:contr} to define a mapping 
$\Gamma:\mathcal{A}\mapsto\mathcal{C}$ that assigns to every admissible 
sequence $\mathcal{A}=(\mathcal{A}_n)$ a new admissible sequence 
$\mathcal{C}=(\mathcal{C}_n)$. This mapping can be thought of as inducing 
a form of dynamics on the space of admissible sequences. If we define the 
value of an admissible sequence $\mathcal{A}$ and the target upper bound 
as
$$
	\mathop{\mathrm{val}}(\mathcal{A}) = 
	\Bigg[
        \sup_{x\in T}
        \sum_{n\ge 0}(2^{n/\alpha}\diam(A_n(x)))^p\Bigg]^{1/p},\quad
	S = 
	\Bigg[
        \sup_{x\in T}
        \sum_{n\ge 0}(2^{n/\alpha}s_n(x))^p\Bigg]^{1/p},
$$
then we have shown in the proof that
$$
	\mathop{\mathrm{val}}(\Gamma(\mathcal{A})) \le
	Ca\mathop{\mathrm{val}}(\mathcal{A}) +
        CS
$$
for a universal constant $C$. If $a$ can be chosen sufficiently 
small that $Ca<1$, then the mapping $\Gamma$ defines a sort of contraction 
on the space of admissible sequences. This ensures that there exists an 
admissible sequence with value $\mathrm{val}(\mathcal{A})\lesssim S$, 
which is the conclusion we seek. This procedure is reminiscent of the 
contraction mapping principle, which is why we refer to Theorem 
\ref{thm:contr} as the contraction principle.

\subsection{Second proof}

The above proof of Theorem \ref{thm:contr} ensures the existence of a good 
admissible partition without directly constructing this partition. This is 
in contrast to the partitioning scheme of \cite{Tal14}, where a good 
admissible partition is explicitly constructed in the proof. We presently 
show that by organizing the proof in a slightly different way, we also
obtain an explicit construction.

\begin{proof}[Second proof of Theorem \ref{thm:contr}]
We construct an increasing sequence of partitions $(\mathcal{B}_n)$ of $T$ 
by induction. First, set $\mathcal{B}_0=\{T\}$. Now suppose partitions 
$\mathcal{B}_0,\ldots,\mathcal{B}_{n-1}$ have already been constructed. 
We first split every set $B_{n-1}\in\mathcal{B}_{n-1}$ into $n$ 
segments
\begin{align*}
	B_n^i &:= \{x\in B_{n-1}:2^{-2i/\alpha}\diam(T)<s_n(x)\le 
	2^{-2(i-1)/\alpha}\diam(T)\} \quad(1\le i<n),\\
	B_n^n &:= \{x\in B_{n-1}:s_n(x) \le 2^{-2(n-1)/\alpha}\diam(T)\},
\end{align*}
and then further subdivide each segment $B_n^i$ into less than $2^{2^n}$ 
pieces $B_n^{ij}$ such that
$$
	\diam(B_n^{ij}) \le 2a\diam(B_{n-1}) +
	2^{1+2/\alpha}s_n(x) +
	2^{1-2(n-1)/\alpha}\diam(T)
	\quad\mbox{for all }x\in B_n^{ij}.
$$
Now let $\mathcal{B}_n=\{B_n^{ij}:B_{n-1}\in\mathcal{B}_{n-1},i\le 
n,j<2^{2^n}\}$.
As $|\mathcal{B}_n|<\prod_{k=1}^n k2^{2^k}<2^{2^{n+2}}$, 
$(\mathcal{B}_n)$ is not itself an admissible sequence. We can however 
easily convert it to an admissible sequence $(\mathcal{A}_n)$ by defining
$\mathcal{A}_0=\mathcal{A}_1=\{T\}$ and $\mathcal{A}_{n+2}=\mathcal{B}_n$.

Now note that by construction, we have
$$
	\diam(A_{n}(x)) \le
	2a\diam(A_{n-1}(x)) +
	2^{1+2/\alpha}s_{n-2}(x) + 2^{1-2(n-3)/\alpha}\diam(T)
$$
whenever $n\ge 3$. Therefore, in the notation of the previous subsection,
$$
        \mathop{\mathrm{val}}(\mathcal{A}) \le
	Ca\mathop{\mathrm{val}}(\mathcal{A})+CS
$$
for a universal constant $C$ depending only on $\alpha$, where we used the 
same argument as in the first proof of Theorem \ref{thm:contr} to absorb 
the $\diam(T)$ term. We now consider two cases. If $Ca\le\frac{1}{2}$, 
say, then we obtain the desired bound 
$\gamma_{\alpha,p}(T)\le\mathop{\mathrm{val}}(\mathcal{A})\le 2CS$ (this 
is the interesting case). On the other hand, if $Ca>\frac{1}{2}$, we 
trivially have $\gamma_{\alpha,p}(T)\le 2Ca\gamma_{\alpha,p}(T)$. Thus the 
conclusion of Theorem \ref{thm:contr} follows.
\end{proof}

The second proof of Theorem \ref{thm:contr} is reminiscent of the 
partitioning scheme of \cite{Tal14} to the extent that an admissible 
sequence is constructed by repeatedly partitioning the index set $T$. In 
contrast to the method of \cite{Tal14}, however, the present approach is 
completely devoid of subtlety: the partitioning at each stage is performed 
in the most naive possible way by breaking up each set arbitrarily into 
pieces of the smallest possible diameter. We will nonetheless see in 
section \ref{sec:mm} that the growth functional machinery of 
\cite{Tal14} can be fully recovered from Theorem~\ref{thm:contr} with a 
remarkably simple proof. In our approach, the growth functional plays no 
role in the partitioning process itself, but will only be used to produce 
controls $s_n(x)$ that yield good \emph{a priori} bounds on the entropy 
numbers $e_n(A)$ for $A\subseteq T$.

\section{Simple illustrations}
\label{sec:examples}

Before we can apply Theorem \ref{thm:contr} in a nontrivial manner, we 
should develop some insight into the meaning of the numbers $s_n(x)$ and 
basic ways in which they can be constructed. To this end, we aim in this 
section to illustrate Theorem \ref{thm:contr} in the simplest cases. All 
results developed here admit more direct proofs, but the present treatment 
is intended to help understand the meaning of Theorem \ref{thm:contr}.

\subsection{Admissible sequences and nets}
\label{sec:adnet}

When the abstract statement of Theorem~\ref{thm:contr} is first 
encountered, it may be far from obvious why the assumption
$$
	e_n(A) \le a\diam(A) + \sup_{x\in A}s_n(x)
$$
is a natural one. The relevance of the numbers $s_n(x)$ can be immediately 
clarified by observing that a canonical choice is already built into the 
definition of $\gamma_{\alpha,p}(T)$.

\begin{lem}
\label{lem:triv}
Let $(\mathcal{A}_n)$ be any admissible sequence of $T$. Then the choice
$s_n(x)=\diam(A_n(x))$ satisfies the assumption of Theorem \ref{thm:contr} 
with $a=0$.
\end{lem}

\begin{proof}
Any set $A\subseteq T$ can be covered by less than $2^{2^{n}}$ sets
$\{A_n(x):x\in A\}$ of diameter at most $\sup_{x\in A}\diam(A_n(x))$.
Thus $e_n(A) \le \sup_{x\in A}\diam(A_n(x))$.
\end{proof}

Of course, with this choice, the conclusion of Theorem \ref{thm:contr} is 
$\gamma_{\alpha,p}(T)\lesssim\gamma_{\alpha,p}(T)$ which is not very 
interesting. Nonetheless, Lemma \ref{lem:triv} explains why bounding 
entropy numbers $e_n(A)$ in terms of controls $s_n(x)$ is entirely 
natural. Moreover, we see that Theorem \ref{thm:contr} can in principle 
always give a sharp bound on $\gamma_{\alpha,p}(T)$.

As an only slightly less trivial example, let us show that the chaining 
functional $\gamma_2^*(T)$ defined in the introduction is always of the 
same order as $\gamma_2(T)$.

\begin{lem}
$\gamma_2^*(T)\asymp\gamma_2(T)$.
\end{lem}

\begin{proof}
As was noted in section \ref{sec:defn}, the inequality
$\gamma_2^*(T)\le\gamma_2(T)$ is trivial. To prove the converse 
inequality, let $T_n$ be arbitrary sets of cardinality $|T_n|<2^{2^n}$, 
and define $s_n(x)=d(x,T_n)$. The definition of entropy numbers instantly 
yields $e_n(A)\le \sup_{x\in A}s_n(x)$. We can therefore apply Theorem 
\ref{thm:contr} with $a=0$ to obtain
$$
	\gamma_2(T) \lesssim
	\sup_{x\in T}\sum_{n\ge 0}2^{n/2}d(x,T_n).
$$
Taking the infimum over all choices of $T_n$ yields the conclusion
$\gamma_2(T)\le\gamma_2^*(T)$.
\end{proof}

So far, we have only used Theorem \ref{thm:contr} with $a=0$ and have not 
exploited the ``contraction'' part of the contraction principle. In the 
next section, we provide a first illustration of an improvement that can 
be achieved by exploiting contraction.

\subsection{A local form of Dudley's inequality}

The most naive bound on $\gamma_2^*(T)$ is obtained by moving the supremum 
in its definition inside the sum. This yields the following result, which 
is known as Dudley's inequality:
$$
	\gamma_2^*(T) \le
	\sum_{n\ge 0}2^{n/2}e_n(T).
$$
Dudley's inequality represents the simplest possible construction where 
each net $T_n$ in the definition of $\gamma_2^*(T)$ is distributed as 
uniformly as possible over the index set $T$. Unfortunately, such a simple 
construction proves to be suboptimal already in some of the simplest 
examples (cf.\ \cite{Tal14,vH16a}). To attain the sharp bound that is 
guaranteed by Theorem \ref{thm:mm}, it is essential to allow for the nets 
$T_n$ to be constructed in a genuinely multiscale fashion. Nonetheless, 
Dudley's inequality is widely used in practice due to the ease with which 
it lends itself to explicit computations.

It is no surprise that Dudley's inequality is trivially recovered by
Theorem \ref{thm:contr}.

\begin{lem}
\label{lem:dudley}
There is a universal constant $C$ depending only on $\alpha$ such that
$$
	\gamma_{\alpha,p}(T) \le
	C\Bigg[
	\sum_{n\ge 0}(2^{n/\alpha}e_n(T))^p
	\Bigg]^{1/p}.
$$
\end{lem}

As $e_n(A)\le e_n(T)$, this follows using $s_n(x)=e_n(T)$ and $a=0$ in 
Theorem \ref{thm:contr}. However, without much additional effort, we can 
do slightly better using a simple application of the ``contraction'' part 
of the contraction principle.

To exploit contraction, we note that if $e_n(A)\le a\diam(A)=r$, then the 
assumption of Theorem \ref{thm:contr} is automatically satisfied; thus 
the numbers $s_n(x)$ only need to control the situation where this 
condition fails. As $A$ 
is contained in a ball of radius $\diam(A)$, this condition essentially 
means that a certain ball of radius $r/a$ can be covered by less than 
$2^{2^n}$ balls of proportional radius $r$, which is a sort of doubling 
condition on the metric space $(T,d)$. Let us consider the 
largest radius of a ball that is centered at a given point $x$ for which 
this doubling condition fails:
$$
	e_n^{a,x}(T) :=
	\sup\{r:e_n(T\cap B(x,r/a))>r\},
$$
where $B(x,r):=\{y\in X:d(x,y)\le r\}$. Then clearly $e_n^{a,x}(T)\le 
e_n(T)$, so that this quantity can be viewed as a local improvement on the 
notion of entropy numbers. We can now use Theorem \ref{thm:contr} to show 
that Dudley's inequality remains valid if we replace the (global) 
entropy numbers by their local counterparts.

\begin{lem}
\label{lem:localdud}
There are universal constants $C,a$ depending only on $\alpha$ such that
$$
	\gamma_{\alpha,p}(T) \le
	C\Bigg[\sup_{x\in T}
	\sum_{n\ge 0}(2^{n/\alpha}e_n^{a,x}(T))^p
	\Bigg]^{1/p}.
$$
\end{lem}

\begin{proof}
Let $n\ge 0$ and $x\in A\subseteq T$.
If $\diam(A)>e_n^{a,x}(T)/a$, then by definition
$$
	e_n(A) \le e_n(T\cap B(x,\diam(A))) \le a\diam(A).
$$
On the other hand, if $\diam(A)\le e_n^{a,x}(T)/a$, then trivially
$$
	e_n(A) \le \diam(A) \le \frac{e_n^{a,x}(T)}{a}.
$$
Thus the assumption of Theorem \ref{thm:contr} holds for any $a>0$ 
with $s_n(x) = e_n^{a,x}(T)/a$. The proof is readily concluded by choosing 
$a$ to be a small universal constant.
\end{proof}

An almost identical proof yields a variant of Lemma \ref{lem:localdud} 
given in \cite[eq.\ (1.9)]{Tal96} that uses a regularized form of 
the local entropy numbers $e_n^{a,x}(T)$.\footnote{
	Consider
	$\tilde e_n^{a,x}(T) :=
	\inf\{a^k\diam(T):\prod_{i=0}^k N(T\cap 
	B(x,a^{i-2}\diam(T)),a^i\diam(T)) < 2^{2^n}\}$,
	where $N(A,\varepsilon)$ is the covering number of $A$ by balls of 
	radius $\varepsilon$. The details are left to the reader.
}
While these bounds can improve on Dudley's inequality in some esoteric 
(ultrametric) examples, they are not particularly useful in practice. 
The reason that Lemma \ref{lem:localdud} is included here is to 
help provide some initial intuition for how one might use the
``contraction'' part of the contraction principle. The real power of the 
contraction principle will however arise when it is combined with the 
interpolation method of \cite{vH16a}.

\subsection{The simplest interpolation estimate}
\label{sec:interp}

As the interpolation method will play a crucial role in the remainder of 
this paper, we must begin by recalling the main idea behind this method. 
The aim of this section is to provide a first illustration of how 
interpolation can be used to generate the controls $s_n(x)$ in Theorem 
\ref{thm:contr} by recovering the main result of \cite{vH16a}.

The interpolation method is based on the following construction. Let 
$f:X\to\mathbb{R}_+\cup\{+\infty\}$ be a given penalty function, and 
define the interpolation functional
$$
	K(t,x) := \inf_{y\in X} \{f(y) + td(x,y)\}.
$$
We will assume for simplicity that the infimum in this definition is 
attained for every $t\ge 0$ and $x\in T$, and denote by $\pi_t(x)$ an 
arbitrary choice of minimizer (if the infimum is not attained we can 
easily extend the results below to work instead with near-minimzers).
We now define for every $t\ge 0$ the interpolation sets
$$
	K_t := \{\pi_t(x):x\in T\}.
$$
The key idea of the interpolation method is that the sets $K_t$ 
provide a multiscale approximation of $T$ that is precisely 
of the form suggested by Theorem \ref{thm:mm}.

\begin{lem}
\label{lem:interp}
For every $a>0$, we have
$$
	\sup_{x\in T}\sum_{n\ge 0}2^{n/\alpha}d(x,\pi_{a2^{n/\alpha}}(x)) 
	\lesssim
	\frac{1}{a}\sup_{x\in T}f(x),
$$
where the universal constant depends only on $\alpha$.
\end{lem}

\begin{proof}
As $0\le K(t,x)\le f(x)$ and 
$K(t,x)-K(s,x) \ge (t-s)d(x,\pi_t(x))$, we have
$$
	\sum_{n\ge 0} a2^{n/\alpha} d(x,\pi_{a2^{n/\alpha}}(x)) \lesssim
	\sum_{n\ge 0} \{K(a2^{n/\alpha},x)-K(a2^{(n-1)/\alpha},x)\}
	\le f(x)
$$
for every $x\in T$ and $a>0$.
\end{proof}

Lemma \ref{lem:interp} provides a natural mechanism to create multiscale 
approximations. However, the approximating sets $K_{a2^{n/\alpha}}$ are 
still continuous, and must therefore be discretized in order to bound the 
chaining functional that appears in Theorem~\ref{thm:mm}. The simplest 
possible way to do this is to distribute each net $T_n$ in 
Theorem~\ref{thm:mm} uniformly over the set $K_{a2^{n/\alpha}}$. This yields the 
basic interpolation bound of \cite{vH16a}.

\begin{thm}
\label{thm:interp}
For every $a>0$, we have
$$
	\gamma_\alpha(T) \lesssim
	\frac{1}{a}\sup_{x\in T}f(x) +
	\sum_{n\ge 0} 2^{n/\alpha}e_n(K_{a2^{n/\alpha}}),
$$
where the universal constant depends only on $\alpha$.
\end{thm}

\begin{proof}
By the definition of entropy numbers, we can choose a set
$T_n$ of cardinality less than $2^{2^n}$ such that
$d(x,T_n)\le 2e_n(K_{a2^{n/\alpha}})$ for every $x\in K_{a2^{n/\alpha}}$.
Then
$$
	e_n(A) \le \sup_{x\in A}d(x,T_n) \le
	\sup_{x\in A}d(x,K_{a2^{n/\alpha}}) +
	2e_n(K_{a2^{n/\alpha}})
$$
for every $A\subseteq T$. We can therefore invoke Theorem 
\ref{thm:contr} with
$a=0$ and $s_n(x)=d(x,K_{a2^{n/\alpha}})+2e_n(K_{a2^{n/\alpha}})$,
and applying Lemma \ref{lem:interp} completes the proof.
\end{proof}

The utility of Theorem \ref{thm:interp} stems from the fact that the 
sets $K_t$ are often much smaller than the index set $T$, so that this 
result provides a major improvement over Dudley's bound. This phenomenon 
is illustrated in various examples in \cite{vH16a}. Nonetheless, there is 
no reason to expect that the particular multiscale construction used here 
should always attain the sharp bound that is guaranteed by Theorem 
\ref{thm:mm}. Indeed, it is shown in \cite[section 3.3]{vH16a} that this 
is not necessarily the case.

There are two potential ways in which Theorem \ref{thm:interp} can result 
in a suboptimal bound. First, the ability of this method to produce 
sufficiently ``thin'' sets $K_t$ relies on a good choice of the penalty 
function $f$. While certain natural choices are considered in 
\cite{vH16a}, the best choice of penalty is not always obvious, and a poor 
choice of penalty will certainly give rise to suboptimal bounds. This is, 
however, not an intrinsic deficiency of the interpolation method.

The fundamental inefficiency of Theorem \ref{thm:interp} lies in the 
discretization of the sets $K_t$. The interpolation method cannot itself 
produce discrete nets: it only reveals a multiscale structure inside 
the index set $T$. To obtain the above result, we naively discretized this 
structure by distributing nets $T_n$ as uniformly as possible over the 
sets $K_{a2^{n/\alpha}}$. While this provides an improvement over 
Dudley's bound, such a uniform discretization can incur a significant 
loss. In general, we should allow once again for a multiscale 
discretization of the sets $K_{a2^{n/\alpha}}$. It is easy to modify the 
above argument to formalize this idea; for example, one can easily show 
that
$$
	\gamma_\alpha(T) \lesssim
	\frac{1}{a}\sup_{x\in T}f(x) + \inf \sup_{x\in T}
	\sum_{n\ge 0}2^{n/\alpha}d(\pi_{a2^{n/\alpha}}(x),T_n),
$$
where the infimum is taken over all nets $T_n$ with
$|T_n|<2^{2^{n}}$. This bound appears to be rather useless, however, as 
the quantity on the right-hand side is just as intractable as the quantity
$\gamma_\alpha(T)$ that we are trying to control in the first place.

The basic insight that gave rise to the results in this paper is that it 
is not actually necessary to construct explicit nets $T_n$ to bound the 
right-hand side of this inequality: it suffices to show that the quantity 
on the right-hand side is significantly smaller than $\gamma_\alpha(T)$.
For example, if we could show that
$$
	\inf \sup_{x\in T}
	\sum_{n\ge 0}2^{n/\alpha}d(\pi_{a2^{n/\alpha}}(x),T_n)
	\lesssim a\gamma_\alpha(T),
$$
then the resulting bound $\gamma_\alpha(T) \lesssim a\gamma_\alpha(T)+ 
\frac{1}{a}\sup_{x\in T}f(x)$ would yield an explicit bound on 
$\gamma_\alpha(T)$ by choosing $a$ to be sufficiently small. Such a bound 
captures quantitatively the idea that the sets $K_t$ are much smaller than 
the index set $T$. The author initially implemented this idea
in a special case (section \ref{sec:geomp}) using the formulation 
described above. It turns out, however, that the same scheme of proof is 
applicable far beyond this specific setting and is in some sense 
canonical. The contraction principle of Theorem~\ref{thm:contr} is nothing 
other than an abstract formulation of this idea that will enable us to
efficiently exploit the interpolation method.

For future reference, we conclude this section by recording a convenient 
observation: the mapping $x\mapsto\pi_t(x)$ can often be chosen to be 
a (nonlinear) projection. This was established in \cite{vH16a} in a more 
restrictive setting.

\begin{lem}
\label{lem:proj}
Suppose that $T=\{x\in X:f(x)\le u\}$. Then $K_t\subseteq T$ for every 
$t\ge 0$, and the minimizers $\pi_t(x)$ may be chosen to satisfy 
$\pi_t(\pi_t(x))=\pi_t(x)$.
\end{lem}

\begin{proof}
As $f(\pi_t(x))\le K(t,x)\le f(x)$, we clearly have $\pi_t(x)\in T$ 
whenever $x\in T$. This shows that $K_t\subseteq T$. Now consider the set
$$
	K_t' := \{x\in T:K(t,x)=f(x)\}.
$$
By construction, if $x\in K_t'$, then we may choose $\pi_t(x)=x$.
If $x\not\in K_t'$, we choose $\pi_t(x)$ to be an arbitrary minimizer.
We claim that $\pi_t(x)\in K_t'$ for every $x\in T$. 

Indeed, suppose $\pi_t(x)\not\in K_t'$. Then there exists $z\in X$ such 
that
$$
	f(z)+td(\pi_t(x),z) < f(\pi_t(x)).
$$
But then we have
\begin{align*}
	K(t,x) &= f(\pi_t(x)) + td(x,\pi_t(x)) \\
	&> f(z) + td(\pi_t(x),z) + td(x,\pi_t(x)) \\
	&\ge f(z) + td(x,z)
\end{align*}
by the triangle inequality. This contradicts the definition of $K(t,x)$.

As $\pi_t(x)\in K_t'$ for every $x\in T$, we have $K_t\subseteq K_t'$. On 
the other hand, as $x=\pi_t(x)$ for every $x\in K_t'$, it follows that 
$K_t=K_t'$ and $\pi_t(\pi_t(x))=\pi_t(x)$.
\end{proof}

\section{Banach lattices and uniform convexity}
\label{sec:geom}

In this section, we encounter our first nontrivial application of the 
contraction principle. We begin by developing in section \ref{sec:geomp} a 
sharper version of a geometric principle that was obtained in 
\cite{vH16a}, resolving a question posed in \cite[Remark 4.4]{vH16a}. We 
will use this principle in section \ref{sec:lattice} to obtain a rather 
general geometric understanding of the behavior of the chaining 
functionals $\gamma_\alpha$ on Banach lattices. In section 
\ref{sec:unifc}, we discuss an analogous result for uniformly convex 
bodies.

\subsection{A geometric principle}
\label{sec:geomp}

Throughout this section, we specialize our general setting to the case 
that $(X,\|\cdot\|)$ is a Banach space and $T\subset X$ is a symmetric 
compact convex set. We let $d(x,y):=\|x-y\|$, and denote the gauge of $T$ 
as $\|x\|_T :=\inf\{s\ge 0:x\in sT\}$. It is natural in the present 
setting to use a power of the gauge as a penalty function in the 
interpolation method: that is, we define
$$
	K(t,x) := \inf_{y\in X}\{\|y\|_T^r+t\|x-y\|\}
$$
for some $r>0$. The existence of minimizers $\pi_t(x)$ for $x\in T$ is 
easily established,\footnote{
	As $K(t,x)\le \|x\|_T^r\le 1$, we may restrict the
	infimum to be taken over the compact set $y\in T$. But
	$\|y\|_T=\sup_{z\in T^\circ}\langle z,y\rangle$ by duality,
	so the gauge is lower-semicontinuous and the inf is attained.
}
and we define as in section \ref{sec:interp} the interpolation sets
$$
	K_t := \{\pi_t(x):x\in T\}.
$$
We would like to impose geometric assumptions on the sets $K_t$ 
that will allow us to obtain tractable bounds on $\gamma_\alpha(T)$. To 
this end, we will prove a sharper form of a useful geometric principle
identified in \cite[Theorem 4.1]{vH16a}.

\begin{thm}
\label{thm:geom}
Let $q\ge 1$ and $L>0$ be given constants, and suppose that
$$
	\|y-z\|_T^q \le Lt\|y-z\|\quad\mbox{for all }y,z\in K_t,~t\ge 0.
$$
Then
$$
	\gamma_\alpha(T)\lesssim
	\begin{dcases}
	L^{1/q}\Bigg[\sum_{n\ge 0}(2^{n/\alpha}e_n(T))^{q/(q-1)}\Bigg]^{(q-1)/q}
	&\quad (q>1),\\
	L\sup_{n\ge 0}2^{n/\alpha}e_n(T)&\quad(q=1),
	\end{dcases}
$$
where the universal constant depends only on $\alpha$.
\end{thm}

The message of this result is that one can improve substantially on 
Dudley's inequality (which is the case $q=\infty$) if the geometric 
condition of Theorem \ref{thm:geom} is satisfied. This condition is one 
manifestation of the idea that the sets $K_t$ are much smaller than $T$: 
under this condition, every small ball in $K_t$ is contained in a 
proportionally scaled-down copy of $T$. Of course, it is not at all 
obvious how to realize this condition, but we will see below that it 
arises very naturally from the interpolation method under suitable 
geometric assumptions on $T$.

For fixed $q>1$, it was shown in \cite[Theorem 4.1]{vH16a} that the 
conclusion of Theorem~\ref{thm:geom} can be deduced from 
Theorem~\ref{thm:interp}. However, this approach has a crucial drawback: 
the constant in the inequality obtained in this manner diverges as 
$q\downarrow 1$. The key improvement provided by Theorem~\ref{thm:geom} is 
that the constant does not depend on $q$, which allows us in particular to 
attain the limiting case $q=1$. The latter is particularly 
interesting, as the so-called Sudakov lower bound
$$
	\gamma_\alpha(T) \ge \sup_{n\ge 0}2^{n/\alpha}e_n(T)
$$
holds trivially for any $T$. Thus the case $q=1$ of 
Theorem~\ref{thm:geom} gives a geometric condition for the Sudakov 
lower bound to be sharp, as conjectured in \cite[Remark~4.4]{vH16a}.
We will encounter in section \ref{sec:lattice} 
an important example where this is the case.

\begin{proof}[Proof of Theorem \ref{thm:geom}]
Let $n\ge 0$ and $A\subseteq T$. We denote by
$$
	A_t := \{\pi_t(x):x\in A\},\qquad\quad
	s(t,A) := \sup_{x\in A}\|x-\pi_t(x)\|
$$
the projection of $A$ on $K_t$ and the associated projection error. 

We first note that the assumption of the theorem implies that
$$
	A_t \subseteq (Lt\diam(A_t))^{1/q}(z+T)
$$
for some $z\in X$. That is, the projection $A_t$ is contained in a 
``shrunk'' copy of $T$. On the other hand, replacing $A_t$ by $A$ only 
costs the projection error:
\begin{align*}
	e_n(A) &\le e_n(A_t) + s(t,A),\\
	\diam(A_t) &\le \diam(A) + 2s(t,A).
\end{align*}
We can therefore estimate
$$
	e_n(A) \le
	(Lt)^{1/q}(\diam(A) + 2s(t,A))^{1/q}e_n(T)
	+ s(t,A).
$$
We apply this bound with $t=a2^{n/\alpha}$. The idea is now
that the interpolation lemma will take care of the projection error, 
while the ``contraction'' part of the contraction principle allows us to 
exploit the shrinkage created by the geometric assumption.

\textbf{Case $q=1$.} In this case, we can estimate
$$
	e_n(A) \le
	LSa\diam(A) + (2LSa+1)s(a2^{n/\alpha},A),\qquad
	S:=\sup_{n\ge 0}2^{n/\alpha}e_n(T).
$$
Applying the contraction principle of Theorem \ref{thm:contr} gives
$$
	\gamma_\alpha(T) \lesssim
	LSa\,\gamma_\alpha(T) +
	(2LSa+1)
	\sup_{x\in T}\sum_{n\ge 0}2^{n/\alpha}\|x-\pi_{a2^{n/\alpha}}(x)\|.
$$
But we can now use the interpolation Lemma 
\ref{lem:interp} to bound the second term as
$$
	\gamma_\alpha(T) \lesssim
	LSa\,\gamma_\alpha(T) +
	LS+\frac{1}{a}.
$$
We conclude by setting $a= C/LS$ for a sufficiently small universal 
constant $C$.

\textbf{Case $q>1$.} The proof is very similar, but now we use Young's 
inequality $uv \le u^p/C^{p/q} + Cv^q$ with $p=q/(q-1)$ to estimate
$$
	e_n(A) \le
	C\diam(A)
	+ (2C+1)s(a2^{n/\alpha},A)
	+ \bigg(\frac{La}{C}\bigg)^{p/q}
	2^{np/\alpha q}e_n(T)^p.
$$
If $C$ is chosen to be a 
sufficiently small universal constant, then the contraction principle
and interpolation lemma give, respectively,
\begin{align*}
	\gamma_\alpha(T) &\lesssim
	\sup_{x\in T}\sum_{n\ge 0}2^{n/\alpha}\|x-\pi_{a2^{n/\alpha}}(x)\|
	+
	(La)^{p/q}
	\sum_{n\ge 0}(2^{n/\alpha}e_n(T))^p \\
	&\lesssim
	\frac{1}{a} + 
	(La)^{p/q}
	\sum_{n\ge 0}(2^{n/\alpha}e_n(T))^p.
\end{align*}
The proof is completed by optimizing over $a>0$.
\end{proof}

Let us note that the choice of $r>0$ in the definition of $K(t,x)$ appears 
nowhere in the statement of proof of Theorem \ref{thm:geom}. The ability 
to choose $r$ will be convenient, however, when we try to verify that the 
assumption of Theorem \ref{thm:geom} is satisfied.

\subsection{Banach lattices}
\label{sec:lattice}

The aim of this section is to show that Theorem \ref{thm:geom} provides a 
rather general understanding of the behavior of $\gamma_\alpha(T)$ on 
Banach lattices. All the relevant background on Banach lattices and their 
geometry can be found in \cite{LT79}.

In the present section, we specialize the setting of the previous section 
to the case where $(X,\|\cdot\|)$ is a Banach lattice and where the 
compact convex set $T\subset X$ is solid, that is, $x\in T$ and $|y|\le 
|x|$ implies $y\in T$. Solidity of $T$ is simply the requirement that 
the gauge $\|\cdot\|_T$ is also a lattice norm (on its domain).

We now introduce a fundamental property that plays an important role in 
the geometry of Banach lattices, cf.\ \cite[section 1.f]{LT79}.

\begin{defn}
Let $q\ge 1$. $T$ satisfies a \emph{lower 
$q$-estimate} with constant $M$ if
$$
	\Bigg[\sum_{i=1}^n\|x_i\|_T^q\Bigg]^{1/q} \le
	M\Bigg\|\sum_{i=1}^n |x_i|\Bigg\|_T
$$
for all $n\ge 1$ and vectors $x_1,\ldots,x_n\in X$.
\end{defn}

We have the following result.

\begin{thm}
\label{thm:lattice}
Let $q\ge 1$. If $T$ satisfies a lower $q$-estimate with constant $M$, 
then
$$
	\gamma_\alpha(T)\lesssim
	\begin{dcases}
	M\Bigg[\sum_{n\ge 0}(2^{n/\alpha}e_n(T))^{q/(q-1)}\Bigg]^{(q-1)/q}
	&\quad (q>1),\\
	M\sup_{n\ge 0}2^{n/\alpha}e_n(T)&\quad(q=1),
	\end{dcases}
$$
where the universal constant depends only on $\alpha$.
\end{thm}

We will prove this theorem by showing that the condition of Theorem 
\ref{thm:geom} is satisfied if we choose $r=q$ in the previous section.
There is a somewhat subtle point, however, that we must take care of 
first. The computations used in our proof rely crucially on the fact that 
a lower $q$-estimate is satisfied with constant $M=1$. However, we did not 
require this special situation to hold in Theorem \ref{thm:lattice}. We 
will therefore make essential use of the observation that any Banach 
lattice that satisfies a lower $q$-estimate admits an equivalent renorming 
whose lower $q$-estimate constant is identically one \cite[Lemma 
1.f.11]{LT79}. Concretely, define the new norm
$$
	\|x\|_{\tilde T} := \sup \Bigg[\sum_{i=1}^n\|x_i\|_T^q\Bigg]^{1/q},
$$
where the supremum is taken over all possible decompositions of $x$ as a 
sum of $n\ge 1$ pairwise disjoint elements $x_1,\ldots,x_n$, and define
$\tilde T:=\{x\in X:\|x\|_{\tilde T}\le 1\}$. It is readily verified 
using \cite[Proposition 1.f.6]{LT79} that 
if $T$ satisfies a lower $q$-estimate with constant $M$, then $\tilde T$ 
satisfies a lower $q$-estimate with constant $1$ and $\tilde T\subseteq 
T\subseteq M\tilde T$. This implies in particular that
$\gamma_\alpha(T)\le M\gamma_\alpha(\tilde T)$ and $e_n(\tilde T)\le 
e_n(T)$, so that we may assume without loss of generality in the proof of 
Theorem \ref{thm:lattice} that $M=1$.

\begin{proof}[Proof of Theorem \ref{thm:lattice}]
We assume without loss of generality that $M=1$, and apply the setting of 
the previous section with $r=q$. Fix $t\ge 0$ and $y,z\in K_t$, and define
$$
	u := (y\wedge z)\vee 0 + (y\vee z)\wedge 0.
$$
The point of this definition is that 
$$
	|y|-|u|=|y-u|\le|y-z|,
$$
as well as the analogous property where the roles of $y$ and $z$ are 
exchanged.

Using that $T$ satisfies a lower $q$-estimate 
with constant one, we obtain
$$
	\|y-u\|_T^q \le \|y\|_T^q-\|u\|_T^q.
$$
On the other hand, Lemma \ref{lem:proj} gives
$$
	\|y\|_T^q = K(t,y) \le \|u\|_T^q + t\|y-u\|.
$$
All the above properties hold if we exchange $y$ and $z$.
We can therefore estimate
\begin{align*}
	\|y-z\|_T^q &\le
	2^{q-1}(\|y-u\|_T^q + \|z-u\|_T^q) \\
	&\le
	2^{q-1}t\,(\|y-u\| + \|z-u\|) \\
	&\le
	2^q t\|y-z\|,
\end{align*}
where we used the triangle inequality, $(a+b)^q\le 2^{q-1}(a^q+b^q)$, and 
that $\|\cdot\|$ is a lattice norm.
The proof is concluded by applying Theorem \ref{thm:geom}.
\end{proof}

An interesting example of Theorem \ref{thm:lattice} is the the following. 
Let $X=\mathbb{R}^d$, let $\|\cdot\|$ be any $1$-unconditional norm (with 
respect to the standard basis), and let $T=B_1^d$ be the unit 
$\ell_1$-ball. It is immediate that the $\ell_1$-norm satisfies a 
$1$-lower estimate with constant one. Theorem \ref{thm:lattice} therefore 
yields
$$
	\gamma_\alpha(B_1^d) \asymp
	\sup_{n\ge 0}2^{n/\alpha}e_n(B_1^d),
$$
that is, Sudakov's lower bound is sharp for the $\ell_1$-ball. In the 
special case where $\alpha=2$ and $\|\cdot\|$ is the Euclidean norm, this 
can be verified by an explicit computation using Theorem \ref{thm:mm}
(or using Theorem \ref{thm:interp}, cf.\ \cite[section 3.2]{vH16a}) 
and simple estimates on the entropy numbers; however, such a computation 
does not explain \emph{why} Sudakov's lower bound turns out to be sharp in 
this setting. Theorem~\ref{thm:lattice} provides a geometric explanation 
of this phenomenon, and extends it to the much more general situation 
where $\|\cdot\|$ is an arbitrary unconditional norm.

We conclude this section with a few remarks.

\begin{rem}
We have shown that Sudakov's inequality is sharp for $B_1^d$ if 
$\|\cdot\|$ is a lattice norm (that is, unconditional with respect to the 
standard basis). It is worth noting that the lattice property is really 
essential for this phenomenon to occur: the analogous result for general 
norms is absolutely false. To see why this must be the case, note that if 
$T$ is the symmetric convex hull of $d$ points in $X$, then we always have 
$T=AB_1^d$ for some linear operator $A:\mathbb{R}^d\to X$. We can 
therefore write $\gamma_\alpha(T,\|\cdot\|)= 
\gamma_\alpha(B_1^d,\|\cdot\|')$ with $\|x\|':=\|Ax\|$. Thus if Sudakov's 
lower bound were sharp for $B_1^d$ when endowed with a general norm, then 
Sudakov's lower bound would be sharp for any symmetric polytope, and 
therefore (by approximation) for every symmetric compact convex set. This 
conclusion is clearly false.
\end{rem}

\begin{rem}
The case $q=1$ of Theorem \ref{thm:lattice} proves to be somewhat 
restrictive. 
Suppose that $\|\cdot\|_T$ satisfies a $1$-lower estimate with constant 
one (as may always be assumed after equivalent renorming). Because of the 
triangle inequality, we must then have the rather strong condition 
$\|x\|_T+\|y\|_T = \|(|x|+|y|)\|_T$. A Banach lattice satisfying this 
condition is called an AL-space. It was shown by Kakutani that such a 
space is always order-isometric to $L^1(\mu)$ for some measure $\mu$ 
\cite[Theorem 1.b.2]{LT79}. Thus $L^1$-balls are 
essentially the only examples for which Theorem \ref{thm:lattice} applies 
with $q=1$. The case $q>1$ is much richer, however, and 
Theorem \ref{thm:lattice} provides a very general tool to understand 
chaining functionals in this setting.
\end{rem}

\begin{rem}
Theorem \ref{thm:lattice} shows that Dudley's inequality can be 
substantially improved for solid sets $T$ that satisfy a nontrivial 
lower $q$-estimate. On the other hand, a solid set $T$ that fails to 
satisfy any nontrivial lower $q$-estimate must contain 
$\ell_\infty^d$-balls of arbitrarily large dimension, cf.\ \cite[Theorem 
1.f.12]{LT79}. For cubes, the majorizing measure theorem and the 
results of \cite{Car81} can be used to show that Dudley's inequality is 
sharp, and that no improvement as in Theorem \ref{thm:lattice} 
can hold in general. Thus Theorem \ref{thm:lattice} is essentially the 
best result of its kind.
\end{rem}

\subsection{Uniformly convex bodies}
\label{sec:unifc}

The lower $q$-estimate property of a Banach lattice is closely related to 
the notion of uniform convexity in general Banach spaces, as is explained 
in \cite[section 1.f]{LT79}. It is therefore not surprising that an 
analogue of Theorem \ref{thm:lattice} holds in a general Banach space when 
$T$ is a uniformly convex body. Unlike the results of the previous 
section, however, this case is already well understood 
\cite[section 4.1]{Tal14}. It will nonetheless be useful to revisit this 
setting in the light of the present paper, as the method that appears in 
the proof will play an essential role in the random matrix problems that 
will be discussed in section \ref{sec:rmt}.

To this end, we return to the setting where $(X,\|\cdot\|)$ is a general 
Banach space and $T\subset X$ is a symmetric compact convex set.

\begin{defn}
\label{defn:ucvx}
Let $q\ge 2$. $T$ is said to be \emph{$q$-convex} with constant $\eta$ if
$$
	\bigg\|\frac{x+y}{2}\bigg\|_T \le 1-\eta\|x-y\|_T^q
$$
for all vectors $x,y\in T$.
\end{defn}

It was shown in \cite[Lemma 4.7]{vH16a} that the assumption of Theorem 
\ref{thm:geom} holds in the present setting with $L=1/2\eta$; the proof of 
this fact is not unlike the one we used in the lattice case. Thus the 
conclusion in the case $q$-convex bodies matches verbatim the one obtained 
for lattices in the previous section. However, in this setting we 
are never near the boundary case of Theorem \ref{thm:geom}, as the 
$q$-convexity property can only hold for $q\ge 2$ (no body is 
more strongly convex than a Euclidean ball).  This means that the 
machinery of this paper is not really needed to establish this result;
it was shown in \cite{vH16a} that it already follows from Theorem 
\ref{thm:interp}.

However, the boundary case reappears if we consider the more general 
chaining functionals $\gamma_{\alpha,p}(T)$ rather than just
$\gamma_\alpha(T)$. For example, the following sharp bound of
\cite[Theorem 4.1.4]{Tal14} cannot be recovered using the methods of 
\cite{vH16a}.

\begin{thm}
\label{thm:ucvx}
Let $q\ge 2$. If $T$ is $q$-convex with constant $\eta$, then
$$
	\gamma_{\alpha,q}(T) \lesssim
	\eta^{-1/q}\sup_{n\ge 0}2^{n/\alpha}e_n(T),
$$
where the universal constant depends only on $\alpha$.
\end{thm}

We will presently give a short proof of this result using the methods of 
this paper in order to highlight a couple of points that arise when 
bounding $\gamma_{\alpha,p}(T)$. Of course, 
one can obtain extensions of both Theorems \ref{thm:lattice} and 
\ref{thm:ucvx} that bound $\gamma_{\alpha,p}(T)$ with general $\alpha>0$ 
and $1\le p\le q$ (not just $p=q$ as in Theorem \ref{thm:ucvx});
as no new ideas arise in this setting, we leave the details to the reader.

In order to bound $\gamma_{\alpha,q}$, we require in principle only a 
minor adaptation of the interpolation method: we simply modify the 
definition of $K(t,x)$ in section \ref{sec:geomp} to
$$
	K(t,x) := \inf_{y\in X}\{\|y\|_T^r + t^q\|x-y\|^q\}.
$$
We denote once again by $\pi_t(x)$ the minimizer in this expression, and 
by $K_t$ the set of minimizers for $x\in T$. The appropriate analogue of 
the interpolation lemma in this setting is obtained by repeating 
verbatim the proof of Lemma \ref{lem:interp}.

\begin{lem}
\label{lem:qinterp}
For every $a>0$, we have
$$
	\sup_{x\in T}\sum_{n\ge 0}
	(2^{n/\alpha}\|x-\pi_{a2^{n/\alpha}}(x)\|)^q
	\lesssim
	\frac{1}{a^q},
$$
where the universal constant depends only on $\alpha$.
\end{lem}

With these simple modifications, we can now essentially follow the same 
scheme of proof as for Theorem \ref{thm:lattice}, replacing the use of the 
lower $q$-estimate by the $q$-convexity property. There is, however, one 
minor issue that requires some care. In the proof of Theorem 
\ref{thm:lattice} (as in the proof of \cite[Lemma 4.7]{vH16a} where the 
assumption of Theorem \ref{thm:geom} is verified for $q$-convex sets), we 
used the fact that $\pi_t(x)$ possesses the projection property of Lemma 
\ref{lem:proj}. This property is however quite special to interpolation 
functionals of the form $\inf\{f(y)+td(x,y)\}$, as it relies crucially on 
the triangle property of the distance. When the distance is raised to a 
power as in the present setting, the projection property no longer holds 
and we must take care to proceed without it. Fortunately, it turns out to 
that the projection property was not really used in an essential way in 
Theorem \ref{thm:lattice} and can easily be avoided.

\begin{proof}[Proof of Theorem \ref{thm:ucvx}]
For $x,y\in T$, the $q$-convexity property can be formulated as
$$
	\bigg\|\frac{x+y}{2}\bigg\|_T \le 
	\max(\|x\|_T,\|y\|_T) -\eta\|x-y\|_T^q
$$
by applying Definition \ref{defn:ucvx} to $x/\gamma$, $y/\gamma$ with
$\gamma=\max(\|x\|_T,\|y\|_T)\le 1$. To exploit this formulation of
$q$-convexity, we will choose $r=1$ in the definition of $K(t,x)$.

Let $n\ge 0$ and $A\subseteq T$. As in the proof of Theorem 
\ref{thm:geom}, we write
$$
	A_t := \{\pi_t(x):x\in A\},\qquad\quad
	s(t,A) := \sup_{x\in A}\|x-\pi_t(x)\|.
$$
Note that $A_t\subseteq T$. If $y=\pi_t(x)$ for $x\in A$, we can estimate
\begin{align*}
	\|y\|_T & \le \|y\|_T + t^q\|x-y\|^q = K(t,x) \\
	&\le
	\|u\|_T + t^q\|x-u\|^q \\
	& \le
	\|u\|_T + 2^{q-1}t^q\|y-u\|^q +
	2^{q-1}t^q\|x-y\|^q \\
	& \le
	\|u\|_T + 2^{q-1}t^q\|y-u\|^q +
	2^{q-1}t^qs(t,A)^q
\end{align*}
for any $u\in X$, where we used the triangle inequality and
$(a+b)^q\le 2^{q-1}(a^q+b^q)$. Thus for any $y,z\in A_t$, choosing
$u:=(y+z)/2$ in the above inequality shows that
$$
	\max(\|y\|_T,\|z\|_T) \le
	\bigg\|\frac{y+z}{2}\bigg\|_T 
	+ 2^{q-1}t^q\bigg\|\frac{y-z}{2}\bigg\|^q +
	2^{q-1}t^qs(t,A)^q.
$$
Applying the $q$-convexity property yields
$$
	\eta\|y-z\|_T^q \le
	2^{-1}t^q\|y-z\|^q +
	2^{q-1}t^qs(t,A)^q.
$$
for all $y,z\in A_t$. Note that this condition is very similar to the 
assumption of Theorem~\ref{thm:geom}, except that an additional projection 
error term appears. The latter is the price we pay for avoiding the
projection property, which does not hold in the present setting. However, 
this additional term introduces no further complications.

The above inequality shows that
$$
	A_t \subseteq 
	\eta^{-1/q}t (\diam(A_t) + 2s(t,A))(z+T)
$$
for some $z\in X$. Proceeding as in the proof of Theorem \ref{thm:geom},
we obtain
$$	e_n(A) \le Sa\diam(A) + (4Sa+1)s(a2^{n/\alpha},A),\qquad
	S := \eta^{-1/q}\sup_{n\ge 0}2^{n/\alpha}e_n(T).
$$
Applying Theorem \ref{thm:contr} and Lemma \ref{lem:qinterp} yields
\begin{align*}
	\gamma_{\alpha,q}(T) &\lesssim
	Sa\,\gamma_{\alpha,q}(T) +
	(4Sa+1)
	\Bigg[\sup_{x\in T}
	\sum_{n\ge 0}(2^{n/\alpha}\|x-\pi_{a2^{n/\alpha}}(x)\|)^q
	\Bigg]^{1/q} \\
	&\lesssim
	Sa\,\gamma_{\alpha,q}(T) +
	S + \frac{1}{a}.
\end{align*}
We conclude by choosing $a=C/S$ for a sufficiently small universal
constant $C$.
\end{proof}

It is also possible to give a proof more in the spirit of Theorem 
\ref{thm:lattice} where we choose $r=q$ in the definition of $K(t,x)$.
In this case, one should replace Definition~\ref{defn:ucvx}
by the following homogeneous form of the $q$-convexity property:
$$
        \bigg\|\frac{x+y}{2}\bigg\|_T^q \le
	\frac{\|x\|_T^q+\|y\|_T^q}{2}
	-\tilde\eta\|x-y\|_T^q
$$
for all $x,y\in X$. It can be shown that this alternative formulation is 
equivalent to that of Definition \ref{defn:ucvx} \cite[Proposition 
7]{BCL94}, and a more careful accounting of the constants (as in 
\cite[Lemma 2.2]{Nao12}) shows that $\tilde\eta\ge c^q\eta$ for a 
universal constant $c$.

\begin{rem}
As was mentioned above, the analogue of Theorem~\ref{thm:lattice} for 
$q$-convex sets was already proved in \cite{vH16a} using only Theorem 
\ref{thm:interp}: one can show in this case that the entropy numbers of 
the interpolation sets $e_n(K_t)$ can be controlled efficiently by the 
entropy numbers $e_n(T)$. It was even shown in \cite{vH16a} by 
explicit computation that Theorem~\ref{thm:interp} yields a sharp bound 
for the $\ell_1$-ball in the special case that $\|\cdot\|$ is the 
Euclidean norm, which is a boundary case of Theorem~\ref{thm:lattice}. 
This is simpler conceptually than the present approach, which relies 
on the contraction principle. One might therefore wonder whether the 
contraction principle is really needed in this setting, or whether it is 
possible that results such as Theorems \ref{thm:lattice} and 
\ref{thm:ucvx} could be recovered from Theorem~\ref{thm:interp} using a 
more efficient argument. We will presently argue that this is not the 
case: the entropy numbers $e_n(K_t)$ are generally too large, so the 
contraction principle is essential to attain sharp bounds.

To this end, consider the following illuminating example. We consider
$X=\mathbb{R}^d$ with the Euclidean distance $\|\cdot\|$, and let 
$T\subset X$ be the ellipsoid defined by
$$
	\|x\|_T^2 =\sum_{k=1}^d k x_k^2.
$$
$T$ is $2$-convex by the parallelogram identity, and Theorem
\ref{thm:ucvx} gives
$$
	\gamma_{2,2}(T) \asymp \sup_{n\ge 0} 2^{n/2}e_n(T)\asymp 1
$$
as $e_n(T)\lesssim 2^{-n/2}$ by the estimates in \cite[section 2.5]{Tal14}.

It is trivial to adapt Theorem \ref{thm:interp} the present setting, which 
yields
$$
	\gamma_{2,2}(T) \lesssim
	\frac{1}{a} +
	\Bigg[
	\sum_{n\ge 0} (2^{n/2}e_n(K_{a2^{n/2}}))^2
	\Bigg]^{1/2} =: S(a).
$$
We claim that this bound cannot recover the correct behavior
of $\gamma_{2,2}(T)$. To see this, we must compute the interpolation sets 
$K_t$. It is particularly convenient in this setting to choose $r=2$ in the 
definition of $K(t,x)$, which is appropriate as explained after the proof 
of Theorem \ref{thm:ucvx}. The advantage of this choice is that 
$K(t,x):=\inf_y\{\|y\|_T^2+t^2\|x-y\|^2\}$ involves minimizing a 
quadratic function, which is trivially accomplished. We readily compute
that $K_t$ is another ellipsoid:
$$
	(\pi_t(x))_k = \frac{t^2}{t^2+k} x_k,\qquad\quad
	\|x\|_{K_t} =	\sum_{k=1}^d \bigg(
	\frac{t^2+k}{t^2}
	\bigg)^2 kx_k^2.
$$
Using the entropy estimate of \cite[Lemma 2.5.4]{Tal14}, we find that
$$
	e_n(K_{a2^{n/2}}) 
	\gtrsim
        \frac{a^2}{a^2+1} 2^{-n/2}
$$
for $2^n\lesssim d$.  It follows that
$$
	S(a) = \frac{1}{a} +
	\Bigg[
	\sum_{n\ge 0} (2^{n/2}e_n(K_{a2^{n/2}}))^2
	\Bigg]^{1/2} \gtrsim
	\frac{1}{a} +
	\frac{a^2}{a^2+1}\sqrt{\log d}
	\gtrsim (\log d)^{1/6}.
$$
We have therefore shown that a sharp bound on $\gamma_{2,2}(T)$ cannot be 
attained by choosing nets that are distributed uniformly on the 
interpolation sets $K_t$, as is done in Theorem \ref{thm:interp}. On the 
other hand, the same interpolation scheme yields a sharp bound when 
combined with the contraction principle in Theorem \ref{thm:ucvx}. This 
example provides an explicit illustration of the assertion made in the 
introduction that the deficiency of Theorem \ref{thm:interp} is not due to 
the interpolation method, but rather due to the fact that the 
interpolation method is being used inefficiently.
\end{rem}

\section{The majorizing measure theorem}
\label{sec:mm}

In the previous sections, we introduced the contraction principle and 
illustrated its utility in combination with the interpolation method in 
several interesting situations. We 
will presently use the same machinery to give a surprisingly simple proof 
of the majorizing measure theorem (Theorem \ref{thm:mm}). With some small 
modifications, this will also allow us to recover the main growth 
functional estimate of \cite{Tal14}. Beside providing simple new proofs of 
these results, the fact that they can be attained at all shows that the 
methods of this paper are not restricted to some special situations, but 
can in fact fully recover the core of the generic chaining theory. 

\subsection{Gaussian processes}
\label{sec:thm1}

Let $(X_x)_{x\in T}$ be a centered Gaussian process, and denote by 
$d(x,y):=(\E|X_x-X_y|^2)^{1/2}$ the associated natural metric on $T$. To 
avoid being distracted by minor measurability issues, let us assume for 
simplicity that the index set $T$ is finite. It is well understood in the 
theory of Gaussian processes that this entails no loss of generality in 
any reasonable situation.

Let us define for any subset $A\subseteq T$ the Gaussian width
$$
	G(A) := \mathbf{E}\bigg[
	\sup_{x\in A}X_x
	\bigg].
$$
The statement of the majorizing measure theorem is that 
$G(T)\asymp\gamma_2(T)$. The upper bound $G(T)\lesssim\gamma_2(T)$ is 
however completely elementary; see \cite[section 2.2]{Tal14} for this 
classical and very simple chaining argument. It is the lower bound 
$\gamma_2(T)\lesssim G(T)$ in the majorizing measure theorem that is a 
deep result. In this section, we will give a simple proof of the 
latter bound using the machinery of this paper.

In its simplest form, the idea that allows us to bound $\gamma_2(T)$ by 
$G(T)$ is clear: we should use $G(T)$ to define the penalty function in 
the interpolation method, and then use Sudakov's inequality for Gaussian 
processes (which gives an upper bound on $e_n(A)$ in terms of $G(A)$) to 
verify the assumption of the contraction principle.
To implement this idea, it will be convenient to define the interpolation 
functional $K(t,x)$ in a somewhat different manner than we did previously: 
we set
$$
	K(t,x):=\inf_{s\ge 0}\{ ts + G(T)-G(B(x,s))\},
$$
where
$$
	B(x,s) := \{y\in T:d(x,y)\le s\}
$$
is the ball in $T$ with radius 
$s$ centered at $x$. As the function $s\mapsto G(B(x,s))$ is 
upper-semicontinuous, the infimum in the 
definition of $K(t,x)$ is attained. Denoting the minimizer as
$s(t,x)\ge 0$, we obtain the following interpolation lemma.

\begin{lem}
\label{lem:ginterp}
For every $a>0$
$$
	\sup_{x\in T}\sum_{n\ge 0}2^{n/2}s(a2^{n/2},x) \lesssim
	\frac{G(T)}{a}.
$$
\end{lem}

The proof is identical to that of Lemma \ref{lem:interp}.

\begin{rem}
It may not be obvious that the present definition of $K(t,x)$ is
an interpolation functional in the sense of section \ref{sec:interp}, 
except in some generalized sense. This is nonetheless the case.
To see why, let $L^\infty(\Omega;T)$ be the space of $T$-valued random
variables endowed with the metric 
$d_\infty(\sigma,\tau):=\|d(\sigma,\tau)\|_\infty$. Then
$$
	G(B(x,s)) = \E\bigg[\sup_{y\in T:d(x,y)\le s}X_y\bigg] =
	\sup_{\tau\in L^\infty(\Omega;T):d_\infty(x,\tau)\le s}
	\E[X_\tau].
$$
Substituting this expression in the definition of $K(t,x)$ and exchanging 
the order of the two infima shows that we can in fact write
$$
	K(t,x) = \inf_{\tau\in L^\infty(\Omega;T)}
	\{G(T) - \E[X_\tau] + td_\infty(x,\tau)\}.
$$
Thus $K(t,x)$ is an interpolation functional, on the space
$(L^\infty(\Omega;T),d_\infty)$ and with penalty function $f(\tau)=G(T) - 
\E[X_\tau]$, of precisely the form given in section \ref{sec:interp}.
While this formulation guides our intuition, it is more convenient 
computationally to work with the definition in terms of $G(B(x,s))$ as 
this will allow us to directly apply inequalities for the suprema of 
Gaussian processes.
\end{rem}

To prove the majorizing measure theorem, we will verify that the condition 
of Theorem \ref{thm:contr} is satisfied with $s_n(x)\lesssim s(a2^{n/2},x)$.
To this end, we must bound the entropy numbers $e_n(A)$ of all subsets
$A\subseteq T$. As our interpolation functional involves the supremum of a 
Gaussian process, this should surely involve Sudakov's inequality. The 
appropriate form for our purposes, which is a straightforward extension of 
Sudakov's inequality, can be found in \cite[Proposition 2.4.9]{Tal14}.

\begin{lem}
\label{lem:supersud}
For $\sigma,b>0$ and $x_1,\ldots,x_n\in T$ such that
$d(x_i,x_j)\ge b$ for $i\ne j$
$$
	\min_{i\le n}G(B(x_i,\sigma)) +
	C_1b\sqrt{\log n} \le
	G(\cup_{i\le n}B(x_i,\sigma)) + C_2\sigma\sqrt{\log n},
$$
where $C_1,C_2$ are universal constants.
\end{lem}

This is in fact a form of the ``growth condition'' that forms the central 
ingredient in the generic chaining theory as developed in \cite{Tal14}. 
One of the advantages of the approach developed in this paper is that it 
makes it possible to bound chaining functionals without engineering such a 
condition, which does not always arise natually in a geometric setting. 
However, in the case of Gaussian processes, the growth condition arises in 
a completely natural manner and is essentially the reason why the 
majorizing measure theorem is true. It therefore seems likely that any 
proof of the majorizing measure theorem must exploit a form of Lemma 
\ref{lem:supersud} at the crucial point in the argument. We will presently 
show that Lemma \ref{lem:supersud} provides a very simple method for 
verifying the assumption of the contraction principle.

\begin{lem}
\label{lem:gausscontr}
For every $n\ge 0$, $A\subseteq T$, and $a>0$, we have
$$
	e_n(A) \lesssim a\diam(A) + (a+1)\sup_{x\in A}s(a2^{n/2},x).
$$
\end{lem}

\begin{proof}
Assume $e_n(A)>0$, otherwise the result is trivial.
By Lemma \ref{lem:packing}, we can find $N=2^{2^n}$ points
$x_1,\ldots,x_N\in A$ such that $d(x_i,x_j)>e_n(A)/2$ for all $i\ne j$.
Let
$$
	\sigma = \sup_{x\in A}s(a2^{n/2},x),\qquad\quad
	r = \diam(A)+\sigma.
$$
Then $\cup_{i\le N}B(x_i,\sigma)\subseteq B(x_k,r)$ for every $k\le N$. We 
can now estimate
\begin{align*}
	G(T) - G(B(x_k,\sigma)) &\le
	G(T) - G(B(x_k,s(a2^{n/2},x_k))) \\ &\le
	K(a2^{n/2},x_k) \\
	&\le a2^{n/2}r + G(T) - G(B(x_k,r)) \\
	&\le a2^{n/2}r + G(T) - G(\cup_{i\le N}B(x_i,\sigma))
\end{align*}
for every $k\le N$. Applying Lemma \ref{lem:supersud} gives
$$
	2^{n/2}e_n(A) \lesssim
	a2^{n/2}r + 2^{n/2}\sigma,
$$
which readily yields the conclusion.
\end{proof}

With this simple lemma in hand, the proof of the lower bound in the 
majorizing measure theorem follows immediately from the contraction 
principle.

\begin{thm}
$\gamma_2(T)\lesssim G(T)$.
\end{thm}

\begin{proof}
The condition of Theorem \ref{thm:contr} is verified by Lemma 
\ref{lem:gausscontr}. It remains to apply Lemma \ref{lem:ginterp} and
choose $a>0$ to be a sufficiently small universal constant.
\end{proof}

\subsection{Growth functionals}
\label{sec:growth}

Now that we have proved the majorizing measure theorem using the approach 
of this paper, it will come as no surprise that the general growth 
functional machinery that forms the foundation of the generic chaining 
theory as developed in \cite{Tal14} can also be recovered by the 
interpolation method. This shows that applicability of the interpolation 
method is not restricted to some special situations, but that it is in 
principle canonical: the generic chaining theory can be fully recovered in 
this manner. In our approach, growth functionals provide one possible 
method for creating the condition of the contraction principle.

In this section, we will modify the proof of the majorizing measure 
theorem to utilize one of the generalized growth functional conditions 
considered in \cite{Tal14}. While the basic idea of the proof is already 
contained in the previous section, this generalization is instructive in 
its own right: it will help clarify the relevance of the ingredients in 
the definition of a growth functional from the present perspective, and 
will also illustrate the use of different interpolation functionals for 
different scales. Of course, the same method of proof admits numerous 
generalizations, including several considered in \cite{Tal14} that can be 
analogously recovered by our methods.

We will work on a general metric space $(T,d)$. Let us begin by stating 
some basic definitions. The first is a notion of well-separated sets
\cite[Definition 2.3.8]{Tal14}.

\begin{defn}
\label{defn:sep}
$H_1,\ldots,H_m\subseteq T$ are 
\emph{$(b,c)$-separated} if there are $x_1,\ldots,x_m,y\in T$
such that $d(x_i,x_j)\ge b$ for all $i\ne j$, and
$d(x_i,y)\le cb$ and $H_i\subseteq B(x_i,b/c)$ for all $i$.
\end{defn}

We also need the basic notion of a functional.

\begin{defn}
A \emph{functional} on $T$ is a map $F$ that assigns to every set
$H\subseteq T$ a number $F(H)\ge 0$ and is increasing, that is,
$F(H)\le F(H')$ if $H\subseteq H'$. A sequence of functionals
$(F_n)_{n\ge 0}$ is \emph{decreasing} if $F_{n+1}(H)\le F_n(H)$ for 
every set $H$.
\end{defn}

We now state the growth condition of \cite[Definition 2.3.10]{Tal14}.

\begin{defn}
\label{defn:growth}
A decreasing sequence of functionals $(F_n)_{n\ge 0}$ satisfies the
\emph{growth condition} with parameters $c,L>0$ if for any $b>0$,
$n\ge 1$ and every collection
$H_1,\ldots,H_N\subseteq T$ of $N=2^{2^n}$ subsets of $T$ that
are $(b,c)$-separated, we have
$$
	F_{n-1}(\cup_{i\le N}H_i) \ge L2^{n/2}b + \min_{i\le N}
	F_n(H_i).
$$
\end{defn}

A minor variation on Lemma \ref{lem:supersud} shows that the choice 
$F_n(H)=G(H)$ satisfies the growth condition provided that the parameter 
$c$ is chosen sufficiently large: that is, the Gaussian width $G(H)$ is a 
growth functional. However, the growth condition as defined above allows 
more flexibility in the design of functionals.

The aim of this section is to prove the following result
\cite[Theorem 2.3.16]{Tal14}.

\begin{thm}
\label{thm:growth}
Suppose that the decreasing sequence of functionals $(F_n)_{n\ge 0}$ 
satisfies the growth condition with parameters $c,L>0$. Then
$$
	\gamma_2(T) \lesssim \frac{c}{L}F_0(T) + \diam(T)
$$
provided that $c\ge c_0$, where $c_0$ is a universal constant.
\end{thm}

In the rest of this section, we fix parameters $c,L>0$ and
a decreasing sequence of functionals $(F_n)_{n\ge 0}$ that
satisfies the growth condition of Definition \ref{sec:growth}.

There are two additional ideas in the proof of Theorem \ref{thm:growth} as 
compared to that of the majorizing measure theorem. First, we have not one 
growth functional $G$, but rather a separate functional $F_n$ for every 
scale. This flexibility introduces more room in the growth condition, 
making it easier to satisfy. The complication that arises is that we have 
to work with multiple interpolation functionals
$$
	K_n(t,x) := \inf_{s\ge 0}\{
	ts+F_0(T)-F_n(B(x,s))
	\}.
$$
However, as $F_n$ is a decreasing sequence of functionals, we readily 
recover a variant of the usual interpolation lemma. We dispose at the same 
time of the minor technical issue that it is unclear whether minimizers in 
the definition of $K_n(t,x)$ exist in the absence of regularity 
assumptions on $F_n$, so we must work with near-minimizers.

\begin{lem}
\label{lem:multinterp}
For every $n\ge 1$, $a>0$ and $x\in T$, choose $s_n^a(x)\ge 0$ such that
\begin{align*}
	K_n(La2^{n/2},x) &\le 
	La2^{n/2}s_n^a(x)+F_0(T)-F_n(B(x,s_n^a(x))) \\ &\le
	K_n(La2^{n/2},x) + 2^{-n}F_0(T).
\end{align*}
Then for every $a>0$
$$
	\sup_{x\in T}
	\sum_{n\ge 1}2^{n/2}s_n^a(x) \lesssim \frac{F_0(T)}{La}.
$$
\end{lem}

\begin{proof}
By definition of $K_{n-1}$ and as $F_n$ is a decreasing sequence,
\begin{align*}
	&2^{-n}F_0(T) +
	K_n(La2^{n/2},x) -
	K_{n-1}(La2^{(n-1)/2},x)
	 \\ &\ge
	(1-2^{-1/2})La2^{n/2}s_n^a(x)
	+F_{n-1}(B(x,s_n^a(x)))
	-F_n(B(x,s_n^a(x)))
	 \\ &\ge
	(1-2^{-1/2})La2^{n/2}s_n^a(x).
\end{align*}
We conclude by summing over $n\ge 1$ and using
$K_n(t,x)\le F_0(T)$ for all $n,t$.
\end{proof}

The second new feature in the proof of Theorem \ref{thm:growth} is that 
the separation condition of Definition \ref{defn:sep} is rather 
restrictive: it requires the sets $H_i$ to have small diameter and all the 
points $x_i$ to be close together. This provides, once again, more room in 
the growth condition of Definition \ref{defn:growth} (as the growth 
condition must only hold for separated sets satisfying these restrictive 
assumptions). However, we will see in the proof of Lemma 
\ref{lem:growthcontr} below that these additional restrictions arise 
essentially for free: if either of these restrictions is violated, the 
condition of the contraction principle is automatically satisfied and 
there is nothing to prove.

\begin{lem}
\label{lem:growthcontr}
Fix $a>0$. Let $s_0(x):=\diam(T)$ and for $n\ge 1$
$$
	s_n(x) :=
	(a+c)s_n^a(x) + 
	\frac{1}{L2^{n/2}}\{
	K_n(La2^{n/2},x)-
	K_{n-1}(La2^{(n-1)/2},x)+
	2^{-n}F_0(T)\}.
$$
Then we have for every $n\ge 0$ and $A\subseteq T$
$$
        e_n(A) \lesssim \bigg(a+\frac{1}{c}\bigg)\diam(A)
		+ \sup_{x\in A}s_n(x).
$$
\end{lem}

\begin{proof}
Assume $n\ge 1$ and $e_n(A)>0$, else the result is trivial.
Let $b=e_n(A)/2$. Lemma \ref{lem:packing} yields $N=2^{2^n}$ points
$x_1,\ldots,x_N\in A$ with $d(x_i,x_j)>b$ for $i\ne j$.
Let
$$
	\sigma = \sup_{x\in A}s_n^a(x),\qquad\quad
	r = \diam(A)+\sigma.
$$
\textbf{Case 1.} If $\sigma> b/c$, then
the conclusion is automatically satisfied as
$$
	e_n(A) < 2c \sup_{x\in A}s_n^a(x) \lesssim 
	\sup_{x\in A} s_n(x).
$$
\textbf{Case 2.} If $\diam(A)> cb$, then
the conclusion is automatically satisfied as
$$
	e_n(A) < \frac{2}{c}\diam(A).
$$
\textbf{Case 3.} If $\sigma\le b/c$ and $\diam(A)\le cb$,
then the sets $H_i=B(x_i,s_n^a(x_i))$, $i=1,\ldots,N$ are
$(b,c)$-separated, so the growth condition can be applied.
We now essentially repeat the proof of Lemma \ref{lem:gausscontr}, except
that we must pay the price
$$
	\Delta_n(x) :=
	K_n(La2^{n/2},x)-K_{n-1}(La2^{(n-1)/2},x)
$$
for switching between two interpolation functionals (notice that 
$\Delta_n(x)\ge 0$ as $F_n$ is a decreasing sequence of functionals).
To be precise, we estimate
\begin{align*}
	&F_0(T) - F_n(H_i) \\
	&\le K_n(La2^{n/2},x_i) + 2^{-n}F_0(T) \\
	&=
	K_{n-1}(La2^{(n-1)/2},x_i) + \Delta_n(x_i)
	+ 2^{-n}F_0(T) \\
	&\le La2^{(n-1)/2}r + F_0(T) - F_{n-1}(B(x_i,r)) + 
	\Delta_n(x_i)
	+ 2^{-n}F_0(T) \\
	&\le La2^{(n-1)/2}r + F_0(T) - F_{n-1}(\cup_{k\le N}H_k) + 
	\Delta_n(x_i)
	+ 2^{-n}F_0(T)
\end{align*}
for every $i\le N$.
Rearranging and applying the growth condition gives
$$
	L2^{n/2}b \le
	F_{n-1}(\cup_{i\le N}H_i) - \min_{i\le N}F_n(H_i) 
	\le
	La2^{(n-1)/2}r +
	\sup_{x\in A}\Delta_n(x) + 2^{-n}F_0(T).
$$
Dividing by $L2^{n/2}$ and using
the definitions of $b,r,\Delta_n$ concludes the proof.
\end{proof}

Note that the quantity $s_n(x)$ in Lemma \ref{lem:growthcontr} has an 
extra term as compared to Lemma \ref{lem:gausscontr}. This additional term 
is the price we pay for switching between different interpolation 
functionals. However, the additional term is completely innocuous: it 
gives rise to a telescoping sum when we apply the contraction principle.

\begin{proof}[Proof of Theorem \ref{thm:growth}]
Applying Lemma \ref{lem:growthcontr} and Theorem \ref{thm:contr} yields
$$
	\gamma_2(T) \lesssim
	\bigg(a+\frac{1}{c}\bigg)\gamma_2(T) +
	\diam(T) + 
	(a+c)\sup_{x\in T}\sum_{n\ge 1}2^{n/2}s_n^a(x)
	+ \frac{F_0(T)}{L},
$$
where we used that $K_n(La2^{n/2},x)\le F_0(T)$ for every $n\ge 1$
and $x\in T$. Thus
$$
	\gamma_2(T) \lesssim
	\bigg(a+\frac{1}{c}\bigg)\gamma_2(T) +
	\frac{1+c/a}{L}F_0(T)
	+ \diam(T)
$$
by Lemma \ref{lem:multinterp}. We can evidently choose a universal 
constant $c_0$ sufficiently large such that the conclusion of the theorem
holds if $c\ge c_0$ and $a=1/c_0$.
\end{proof}

\section{Dimension-free bounds on random matrices}
\label{sec:rmt}

As was stated in the introduction, there are numerous challenging 
probabilistic problems that remain unsolved due to the lack of 
understanding of how to control the supremum of some concrete Gaussian 
process. Such problems arise routinely, for example, in the study of 
structured random matrices \cite{RV08,vH16b,vH17}, whose fine properties 
fall outside the reach of classical methods of random matrix theory. 
Concrete problems of this kind constitute a particularly interesting case 
study for the control of inhomogeneous random processes, and provide 
concrete motivation for the development of new methods to control chaining 
functionals.

Of particular interest in the setting of structured random matrices are 
dimension-free bounds on matrix norms. Such bounds cannot be obtained by 
classical methods of random matrix theory such as the moment method, which 
are inherently dimension-dependent. This is explained in detail 
\cite{vH16b,vH17} in the context of a tantalizing conjecture on Gaussian 
random matrices due to R.\ Lata{\l}a. In this section, we make further 
progress in this direction by developing a closely related result: a 
dimension-free analogue of a well-known result of M.\ Rudelson 
\cite{Rud96}. The proof provides another illustration of the utility of 
the contraction principle.

\subsection{Statement of results}

Throughout this section, let $A_1,\ldots,A_m\in\mathbb{R}^{d\times d}$ be 
nonrandom symmetric matrices, and let $g_1,\ldots,g_m$ be independent 
standard Gaussian variables. We are interested in bounding matrix norms of 
the random matrix
$$
	X = \sum_{k=1}^m g_kA_k
$$
in terms of the coefficients $A_k$. A well-known result of M.\ 
Rudelson \cite{Rud96}, which was proved using a generic chaining 
construction (see also \cite[section 16.7]{Tal14}), states that
$$
	\mathbf{E}\|X\|\lesssim 
	\Bigg\|\sum_{k=1}^m A_k^2\Bigg\|^{1/2}\sqrt{\log(m+1)}
$$
in the important special case where each $A_k=x_kx_k^*$ has rank one (here 
and below $\|\cdot\|$ denotes the spectral norm of a matrix). Due to the 
rank-one assumption, the matrices $A_k$ act nontrivially only on the 
$m$-dimensional subspace of $\mathbb{R}^d$ spanned by the vectors 
$x_1,\ldots,x_m$, so that the above bound is overtly dimension-dependent.
This dimension-dependence is not expected to be sharp when different 
vectors $x_k$ possess substantially different scales. Unfortunately, the 
dependence on dimension arises in an apparently essential manner 
in the approach of \cite{Rud96}. We will see in the sequel that the 
contraction principle makes it possible to avoid this inefficiency.
For example, we can obtain the following dimension-free form of 
Rudelson's bound.

\begin{thm}
\label{thm:dimfreerud}
Suppose that each $A_k=x_kx_k^*$ has rank one. Then
$$
	\mathbf{E}\|X\|\lesssim
	\Bigg\|\sum_{k=1}^m A_k^2\log(k+1)\Bigg\|^{1/2}.
$$
\end{thm}

\begin{rem}
\label{rem:khin}
The generic chaining approach to Rudelson's dimension-dependent bound is 
essentially made obsolete by a much simpler and more general approach 
using the noncommutative Khintchine inequality of Lust-Piquard and Pisier 
\cite{Rud99}. The latter shows that an analogue of Rudelson's bound 
actually holds without any assumption on the matrices $A_k$ (that is, the 
rank-one assumption is not needed); see \cite{vH17} for an 
elementary proof. However, it does not appear that such an approach could 
ever produce a dimension-free bound as in Theorem \ref{thm:dimfreerud}, as 
it relies crucially on the moment method of random matrix theory which is 
inherently dimension-dependent in nature \cite{vH16b}. In addition, the 
moment method is useless for bounding operator norms other than the 
spectral norm, which is important for applications in functional analysis
\cite{GR07,GMPT08,RV08}. Chaining methods appear to be essential for 
addressing problems of this kind that are out of reach of classical random 
matrix theory.
\end{rem}

Theorem \ref{thm:dimfreerud} arises as a special case of a much more 
general result that is of broader interest, and that clarifies the 
geometric structure behind the results of this section. In the remainder 
of this section, we will fix a symmetric compact convex set 
$B\subset\mathbb{R}^d$ that is $2$-convex with constant $\eta$ in the 
sense of Definition \ref{defn:ucvx}. We will be interested in controlling 
$\sup_{v\in T}\langle v,Xv\rangle$ for $T\subseteq B$. When $T=B=B_2^d$ is 
the Euclidean ball, this is simply the largest eigenvalue of $X$ which is 
readily related to the spectral norm. However, we allow in general to 
consider any subset $T\subseteq B$. In addition, following 
\cite{GR07,GMPT08} we can consider any $2$-convex ball $B$ instead of the 
Euclidean ball, which will present no additional complications in the 
proofs.

As $X$ is a Gaussian random matrix, clearly $v\mapsto\langle v,Xv\rangle$
is a centered Gaussian process. It therefore suffices by Theorem 
\ref{thm:mm} to bound the right-hand side of
$$
	\mathbf{E}\bigg[\sup_{v\in T}\langle v,Xv\rangle\bigg]
	\asymp \gamma_2(T,d),
$$
where the natural distance $d(v,w)$ is given by
$$
	d(v,w) := [\mathbf{E}|\langle v,Xv\rangle-
	\langle w,Xw\rangle|^2]^{1/2} =
	\Bigg[\sum_{k=1}^m \langle v+w,A_k(v-w)\rangle^2\Bigg]^{1/2}.
$$
We will also define for $v,z\in\mathbb{R}^d$
$$
	\|v\|_z := 
	\Bigg[\sum_{k=1}^m \langle z,A_kv\rangle^2\Bigg]^{1/2},
	\qquad\quad
        \nnn{v} := \Bigg[
        \sum_{k=1}^m \langle v,A_kv\rangle^2\Bigg]^{1/4}.
$$
The main result of this section is the 
following, which could be viewed as a sort of Gordon embedding 
theorem \cite[Theorem 16.9.1]{Tal14} for structured random matrices.

\begin{thm}
\label{thm:gordon}
Suppose $A_1,\ldots,A_m$ are positive semidefinite. Then
for any $T\subseteq B$
$$
	\mathbf{E}\bigg[\sup_{v\in T}\langle v,Xv\rangle\bigg] \lesssim
	\frac{1}{\sqrt{\eta}}
	\Bigg[
	\sup_{v\in T}
	\sum_{n\ge 0}(2^{n/2}e_n(B,\|\cdot\|_v))^2
	\Bigg]^{1/2} +
	\gamma_{4,2}(T,\nnn{\cdot})^2.
$$
\end{thm}

When Theorem \ref{thm:gordon} is specialized to the case $T=B=B_2^d$, 
we obtain the following bound on the spectral norm of $X$ from which
Theorem \ref{thm:dimfreerud} follows easily.

\begin{cor}
\label{cor:supernck}
Suppose that $A_1,\ldots,A_m$ are positive semidefinite. Then
$$
	\mathbf{E}\|X\| \lesssim
	\Bigg\|\sum_{k=1}^m A_k^2 \Bigg\|^{1/2}
	+ \sup_{n\ge 0}2^{n/2}e_n(B_2^d,\nnn{\cdot})^2.
$$
\end{cor}

The assumption that the matrices $A_k$ are positive semidefinite is a 
natural relaxation of the rank-one assumption in Rudelson's approach 
\cite{Rud96}. This assumption ensures that $\nnn{\cdot}$ is 
a norm. Whether the positive semidefinite assumption can be weakened in 
Theorem \ref{thm:gordon} and Corollary \ref{cor:supernck} is a tantalizing 
question. Indeed, the abovementioned conjecture of Lata{\l}a \cite{vH16b} 
would follow if Corollary \ref{cor:supernck} were to hold for matrices 
$A_k$ that are not positive semidefinite. While one can partially adapt 
the proof of Theorem \ref{thm:gordon} to general $A_k$, significant loss 
is incurred in the resulting bounds. These issues will be further 
discussed in section \ref{sec:disc} below.

\subsection{Proof of Theorem \ref{thm:gordon}}

We will assume throughout this section that the matrices $A_1,\ldots,A_m$ 
are positive semidefinite. This implies, in particular, that $\nnn{\cdot}$ 
is a norm and that $\|v\|_z\le \nnn{v}\nnn{z}$ by Cauchy-Schwarz.

Let us begin by explaining the basic geometric idea behind the proof 
through a back-of-the-envelope computation. Note that
$$
	d(y,z) = \|y-z\|_{y+z} \le 2\|y-z\|_x + 
	\nnn{y-z}(\nnn{y-x}+\nnn{z-x})
$$
by the triangle inequality and Cauchy-Schwarz. Thus
$$
	\diam(A,d) \le 2\diam(A,\|\cdot\|_x) + 2\diam(A,\nnn{\cdot})^2
$$
for any $A\subseteq T$ and $x\in A$.
This suggests we might try to bound $\gamma_2(T,d)$ by the sum of two 
terms, one of the form $\sup_{x\in T}\gamma_2(T,\|\cdot\|_x)$ and another 
of the form $\gamma_2(T,\nnn{\cdot}^2)=\gamma_{4,2}(T,\nnn{\cdot})^2$. If 
that were possible, we would obtain a result far better than Theorem 
\ref{thm:gordon}. The problem, however, lies with the first term: 
a direct 
application of the contraction principle yields not 
$\sup_{x\in T}\gamma_2(T,\|\cdot\|_x)$, but rather 
$$
	\inf_{(\mathcal{A}_n)}
	\sup_{x\in T}\sum_{n\ge 0}2^{n/2}\diam(A_n(x),\|\cdot\|_x).
$$
The latter could be much larger than $\sup_{x\in T}\gamma_2(T,\|\cdot\|_x)$:
here a single admissible sequence $(\mathcal{A}_n)$ 
must control simultaneously every norm $\|\cdot\|_x$, while in the 
definition of $\sup_{x\in T}\gamma_2(T,\|\cdot\|_x)$ each norm is 
controlled by its own admissible sequence. The remarkable aspect of 
Theorem \ref{thm:gordon} is that by exploiting the contraction 
theorem and 2-convexity of $B\supseteq T$, we will nonetheless achieve the 
same upper bound as would be obtained if we were to control
$\sup_{x\in T}\gamma_2(B,\|\cdot\|_x)$ using Theorem \ref{thm:geom}.

We now proceed with the details of the proof. To exploit $2$-convexity, it 
will be useful to replace the natural metric $d$ by a regularized form
$$
	\tilde d(v,w) := d(v,w) + \nnn{v-w}^2.
$$
While $\tilde d$ is not a metric, it is a quasi-metric (the triangle 
inequality holds up to a multiplicative constant). This will suffice for 
all our purposes; in particular, it is readily verified that the proof of 
the contraction Theorem \ref{thm:contr} holds verbatim in a quasi-metric 
space up to the value of the universal constant. We will use this 
observation in the sequel without further comment.
The advantage of $\tilde d$, as opposed to the natural metric, is that it 
behaves in some sense like a norm.

\begin{lem}
\label{lem:quasi}
For every $v,w,z\in\mathbb{R}^d$, we have:
\begin{enumerate}[a.]
\item $\tilde d(v,w) \le 2(\tilde d(v,z)+\tilde d(z,w))$.
\item $\tilde d(v,\frac{1}{2}(v+w)) \le \frac{1}{2}\tilde d(v,w)$.
\end{enumerate}
\end{lem}

\begin{proof}
The first claim follows from the triangle inequality and
$(a+b)^2\le 2(a^2+b^2)$. To prove the second claim, note that
we can write
$$
	v - \tfrac{1}{2}(v+w) = \tfrac{1}{2}(v-w),\qquad
	v + \tfrac{1}{2}(v+w) = \tfrac{1}{2}(v-w) + (v+w).
$$
Therefore
\begin{align*}
	\tilde d(v,\tfrac{1}{2}(v+w)) &=
	\tfrac{1}{2}
	\|\tfrac{1}{2}(v-w) + (v+w)\|_{v-w}
	+
	\tfrac{1}{4}\nnn{v-w}^2 \\
	&\le
	\tfrac{1}{2}(\|v+w\|_{v-w} + \nnn{v-w}^2) =
	\tfrac{1}{2}\tilde d(v,w),
\end{align*}
where we used the triangle inequality.
\end{proof}

We now define the interpolation functional
$$
	K(t,x) := \inf_{y\in\mathbb{R}^d}\{
	\|y\|_B + t\tilde d(x,y)
	\},
$$
and as usual we let $\pi_t(x)$ be a minimizer in this 
expression. Due to the second property of Lemma \ref{lem:quasi} (which 
was engineered precisely for this purpose), we can control the shrinkage 
of interpolation sets as in the proof of Theorem \ref{thm:ucvx}.

\begin{lem}
\label{lem:gordonshrink}
Let $t\ge 0$ and $A\subseteq T$. Then $A_t :=\{\pi_t(x):x\in A\}$
satisfies
$$
	A_t \subseteq \frac{L\sqrt{t}}{\sqrt{\eta}}
	\bigg\{\diam(A,\tilde d)+\sup_{x\in A}\tilde d(x,\pi_t(x))
	\bigg\}^{1/2}
	(z+B)
$$
for some point in $z\in\mathbb{R}^d$,
where $L$ is a universal constant.
\end{lem}

\begin{proof}
Let $x\in A$ and $y=\pi_t(x)$. Then
\begin{align*}
        \|y\|_B \le K(t,x) \le
        \|u\|_B + 2t(\tilde d(x,y) + \tilde d(y,u))
\end{align*}
for any $u\in\mathbb{R}^d$ by the definition of the interpolation 
functional and the first property of Lemma \ref{lem:quasi}.
Therefore, we have for every $y,z\in A_t$ and $u\in\mathbb{R}^d$
$$
        \max(\|y\|_B,\|z\|_B) \le
        \|u\|_B 
	+ 2t \max(\tilde d(y,u),\tilde d(z,u))
	+ 2t \sup_{x\in A}\tilde d(x,\pi_t(x)).
$$
If we choose $u=\tfrac{1}{2}(y+z)$, then we obtain 
$$
        \max(\|y\|_B,\|z\|_B) \le
	\bigg\|\frac{y+z}{2}\bigg\|_B
	+ t\tilde d(y,z)
	+ 2t \sup_{x\in A}\tilde d(x,\pi_t(x))
$$
using the second property of Lemma \ref{lem:quasi}. In particular,
$$
	\eta\|y-z\|_B^2 \le
	t\tilde d(y,z)
	+ 2t \sup_{x\in A}\tilde d(x,\pi_t(x))
$$
for all $y,z\in A_t$ by $2$-convexity of $B$. It follows that
$$
	\diam(A_t,\|\cdot\|_B) \le 
	\frac{\sqrt{t}}{\sqrt{\eta}}
	\bigg\{
	\diam(A_t,\tilde d)
	+ 2\sup_{x\in A}\tilde d(x,\pi_t(x))
	\bigg\}^{1/2}.
$$
It remains to note that
$\diam(A_t,\tilde d)\le 4 \diam(A,\tilde d) + 8\sup_{x\in A}\tilde 
d(x,\pi_t(x))$.
\end{proof}

We now arrive at the main step in the proof of Theorem \ref{thm:gordon}:
we must verify the assumption of the contraction principle.

\begin{lem}
\label{lem:gordoncontr}
Let $(\mathcal{C}_n)$ be an admissible sequence of $T$ and
$a,b>0$. Then 
$$
	e_n(A,\tilde d) \lesssim
	b\diam(A,\tilde d) + \sup_{x\in A} s_n(x)
$$
for every $n\ge 1$ and $A\subseteq T$, where
$$
	s_n(x) :=
	(b+1)\tilde d(x,\pi_{a2^{n/2}}(x)) +
	\frac{a2^{n/2}}{b\eta} e_{n-1}(B,\|\cdot\|_x)^2 +
	\diam(C_{n-1}(x),\nnn{\cdot})^2.
$$
\end{lem}

\begin{proof}
Fix $n\ge 1$ and $A\subseteq T$. For every set
$C\in\mathcal{C}_{n-1}$, define
$$
	A_{a2^{n/2}}^C := \{\pi_{a2^{n/2}}(x):x\in A\cap C\}
$$
and choose an arbitrary point $x_C\in A\cap C$. Now choose, for every 
$C\in\mathcal{C}_{n-1}$, a net $T_{n-1}^C\subseteq A_{a2^{n/2}}^C$ of 
cardinality less than $2^{2^{n-1}}$ such that
$$
	\inf_{z\in T_{n-1}^C}\|y-z\|_{x_C} \le
	4e_{n-1}(A_{a2^{n/2}}^C,\|\cdot\|_{x_C})
	\quad\mbox{for all }y\in A_{a2^{n/2}}^C.
$$
Then $T_n := \bigcup_{C\in\mathcal{C}_{n-1}}T_{n-1}^C$
has cardinality less than $2^{2^n}$. It remains to show that
$$
	\sup_{x\in A}
	\tilde d(x,T_n) \lesssim b\diam(A,\tilde d)+\sup_{x\in A}s_n(x),
$$
which concludes the proof.

To this end, fix $C\in\mathcal{C}_{n-1}$ and
$x\in A\cap C$, and choose $z\in T_{n-1}^C$ such that
$$
	\|\pi_{a2^{n/2}}(x)-z\|_{x_C} \le
	4e_{n-1}(A_{a2^{n/2}}^C,\|\cdot\|_{x_C}).
$$
We can estimate
\begin{align*}
	\tilde d(x,T_n) &\le 2\tilde d(x,\pi_{a2^{n/2}}(x)) +
	2\tilde d(\pi_{a2^{n/2}}(x),T_n)\\
	&\le 2\tilde d(x,\pi_{a2^{n/2}}(x)) +
	2\tilde d(\pi_{a2^{n/2}}(x),z) \\
	&\le 2\tilde d(x,\pi_{a2^{n/2}}(x)) +
	4\|\pi_{a2^{n/2}}(x)-z\|_{x_C} \\
	&\quad + 
	2\nnn{\pi_{a2^{n/2}}(x)-z}(
	\nnn{\pi_{a2^{n/2}}(x)-x_C}+\nnn{z-x_C}
	) \\
	&\quad + \nnn{\pi_{a2^{n/2}}(x)-z}^2.
\end{align*}
As $z\in A_{a2^{n/2}}^C$ by construction, there is a point $x'\in A\cap C$
such that $z=\pi_{a2^{n/2}}(x')$. We therefore obtain, using that
$\nnn{v-w}^2\le\tilde d(v,w)$,
\begin{align*}
	\nnn{\pi_{a2^{n/2}}(x)-z} &\le
	\nnn{x-\pi_{a2^{n/2}}(x)}+\nnn{x-x'}+
	\nnn{x'-\pi_{a2^{n/2}}(x')}
	\\
	&\le
	2\sup_{v\in A}\tilde d(v,\pi_{a2^{n/2}}(v))^{1/2} +
	\diam(C,\nnn{\cdot}).
\end{align*}
Similarly, we can estimate
$$
	\nnn{\pi_{a2^{n/2}}(x)-x_C}  +
	\nnn{z-x_C}
	\le
	2\sup_{v\in A}\tilde d(v,\pi_{a2^{n/2}}(v))^{1/2} +
	2\diam(C,\nnn{\cdot}).	
$$
Putting together the above estimates, we obtain
$$
	\tilde d(x,T_n) \lesssim
	\sup_{v\in A}\tilde d(v,\pi_{a2^{n/2}}(v)) +
	e_{n-1}(A_{a2^{n/2}}^C,\|\cdot\|_{x_C}) +
	\diam(C,\nnn{\cdot})^2
$$
for every $x\in A\cap C$. We now note that by Lemma 
\ref{lem:gordonshrink},
\begin{align*}
	&e_{n-1}(A_{a2^{n/2}}^C,\|\cdot\|_{x_C}) \\ 
	&\lesssim
	\frac{\sqrt{a2^{n/2}}}{\sqrt{\eta}}
	\bigg\{
	\diam(A,\tilde d) + \sup_{v\in A}\tilde d(v,\pi_{a2^{n/2}}(v))
	\Bigg\}^{1/2} 
	e_{n-1}(B,\|\cdot\|_{x_C}) 
	\\ &\lesssim
	\frac{a2^{n/2}}{b\eta}
	\sup_{v\in A}
	e_{n-1}(B,\|\cdot\|_{v})^2 +
	b\diam(A,\tilde d) + b\sup_{v\in A}\tilde d(v,\pi_{a2^{n/2}}(v)).
\end{align*}
As $x\in A\cap C_{n-1}(x)$ for every $x\in A$, we have shown that
\begin{multline*}
	\sup_{x\in A}\tilde d(x,T_n) \lesssim
        b\diam(A,\tilde d)  +
	(b+1)\sup_{v\in A}\tilde d(v,\pi_{a2^{n/2}}(v)) \\ +
        \frac{a2^{n/2}}{b\eta}
        \sup_{v\in A}
        e_{n-1}(B,\|\cdot\|_{v})^2 +
	\sup_{v\in A}
	\diam(C_{n-1}(v),\nnn{\cdot})^2.
\end{multline*}
The proof is concluded using
$\sup_v a_1(v) + \sup_v a_2(v) + \sup_v a_3(v) \le
3\sup_v(a_1(v)+a_2(v)+a_3(v))$ for any nonnegative functions 
$a_1(v),a_2(v),a_3(v)\ge 0$.
\end{proof}

\begin{rem}
\label{rem:propercov}
We used above the standard fact that for any metric space $(X,d)$ and 
$T\subseteq X$, there is a net $T_n\subseteq T$ with $|T_n|<2^{2^n}$
so that $\sup_{x\in T}d(x,T_n)\le 4e_n(T,d)$. We recall 
the proof for completeness. The definition of entropy numbers guarantees 
the existence of a net $S_n\subseteq X$ with $|S_n|<2^{2^n}$ 
so that $\sup_{x\in T}d(x,S_n)\le 2e_n(T,d)$, but $S_n$ need not be a 
subset of $T$. For every point $z\in S_n$, choose $z'\in T$ such that 
$d(z,z')\le 2e_n(T,d)$, and let $T_n\subseteq T$ be the collection of 
points thus constructed. Then $d(x,T_n) \le d(x,S_n) + d(S_n,T_n) \le 
4e_n(T,d)$ for every $x\in T$ as desired. The fact that one can choose the 
net $T_n$ to be a subset of $T$ rather than of $X$ was essential in the 
above proof in order to ensure that $T_{n-1}^C\subseteq A_{a2^{n/2}}^C$.
\end{rem}

We can now complete the proof of Theorem \ref{thm:gordon}.

\begin{proof}[Proof of Theorem \ref{thm:gordon}]
By Theorem \ref{thm:mm}, we have
$$
	\mathbf{E}\bigg[\sup_{v\in T}\langle v,Xv\rangle\bigg]
	\lesssim
	\gamma_2(T,d) \le \gamma_2(T,\tilde d).
$$
Fix $a,b>0$ and an admissible sequence $(\mathcal{C}_n)$ of $T$. Then
\begin{multline*}
	\gamma_2(T,\tilde d) \lesssim
	b\gamma_2(T,\tilde d) +
	\diam(T,\tilde d) +
	(b+1)\sup_{x\in T}\sum_{n\ge 1}2^{n/2}\tilde d(x,\pi_{a2^{n/2}}(x))
	\\
	+ \frac{a}{b\eta}\sup_{x\in T}\sum_{n\ge 1}
	(2^{n/2}e_{n-1}(B,\|\cdot\|_x))^2
	+ \sup_{x\in T}\sum_{n\ge 1}2^{n/2}
	\diam(C_{n-1}(x),\nnn{\cdot})^2
\end{multline*}
by Theorem \ref{thm:contr}, where we used Lemma \ref{lem:gordoncontr}
to define $s_n(x)$ for $n\ge 1$ and the trivial choice 
$s_0(x)=\diam(T,\tilde d)$. Choosing $b$ to be a sufficiently small 
universal constant and applying the interpolation Lemma \ref{lem:interp}
gives
\begin{align*}
	\gamma_2(T,\tilde d) \lesssim
	\diam(T,\tilde d) &+
	\frac{1}{a} +
	\frac{a}{\eta}\sup_{x\in T}\sum_{n\ge 0}
        (2^{n/2}e_n(B,\|\cdot\|_x))^2 \\
        &+ \sup_{x\in T}\sum_{n\ge 0}2^{n/2}
        \diam(C_n(x),\nnn{\cdot})^2.
\end{align*}
Optimizing over $a$ and over admissible sequences $(\mathcal{C}_n)$ of $T$
yields
$$
	\gamma_2(T,\tilde d)\lesssim
	\diam(T,\tilde d)+
	\frac{1}{\sqrt{\eta}}
	\Bigg[\sup_{x\in T}\sum_{n\ge 0}
        (2^{n/2}e_n(B,\|\cdot\|_x))^2\Bigg]^{1/2}
	+ \gamma_{4,2}(T,\nnn{\cdot})^2.
$$
It remains to note that as $\diam(T,\tilde d)\le
2\diam(B,\|\cdot\|_x) + 2\diam(T,\nnn{\cdot})^2$ for any $x\in T$,
the first term can be absorbed in the remaining two.
\end{proof}

\subsection{Proof of Corollary \ref{cor:supernck} and Theorem \ref{thm:dimfreerud}}

Using Theorem \ref{thm:gordon}, the proof of Corollary \ref{cor:supernck} 
follows from classical entropy estimates for ellipsoids.

\begin{proof}[Proof of Corollary \ref{cor:supernck}]
Note that for any $v\in\mathbb{R}^d$, the norm $\|\cdot\|_v$ is 
a Euclidean norm defined by the inner product $\langle x,y\rangle_v := 
\langle x,\Sigma_vy\rangle$ with $\Sigma_v := \sum_{k=1}^mA_kvv^*A_k$.
Thus $e_n(B_2^d,\|\cdot\|_v)$ are entropy numbers of ellipsoids in Hilbert 
space, which are well understood. Using the entropy estimates in
\cite[section 2.5]{Tal14}, we readily obtain
$$
	\sum_{n\ge 0} (2^{n/2}e_n(B_2^d,\|\cdot\|_v))^2 \asymp
	\mathrm{Tr}[\Sigma_v] =
	\Bigg\langle v,\Bigg(\sum_{k=1}^m A_k^2\Bigg)v\Bigg\rangle.
$$
In particular, we obtain
$$
	\Bigg[\sup_{v\in B_2^d}
	\sum_{n\ge 0} (2^{n/2}e_n(B_2^d,\|\cdot\|_v))^2\Bigg]^{1/2}
	\asymp \Bigg\|\sum_{k=1}^m A_k^2\Bigg\|^{1/2}.
$$
On the other hand, by Theorem \ref{thm:ucvx}, we have
$$
	\gamma_{4,2}(B_2^d,\nnn{\cdot}) \asymp
	\sup_{n\ge 0}2^{n/4}e_n(B_2^d,\nnn{\cdot}).
$$
Thus Theorem \ref{thm:gordon} implies
$$
	\mathbf{E}\bigg[\sup_{v\in B_2^d}\langle v,Xv\rangle\bigg]
	\lesssim
	\Bigg\|\sum_{k=1}^m A_k^2\Bigg\|^{1/2} +
	\sup_{n\ge 0}2^{n/2}e_n(B_2^d,\nnn{\cdot})^2.
$$
It remains to note that
$$
	\|X\| = \sup_{v\in B_2^d}|\langle v,Xv\rangle|
	\le \sup_{v\in B_2^d}\langle v,Xv\rangle +
	\sup_{v\in B_2^d}\langle v,(-X)v\rangle
$$
and that $X$ and $-X$ have the same distribution.
\end{proof}

To deduce Theorem \ref{thm:dimfreerud} from Corollary 
\ref{cor:supernck}, we need to estimate the entropy numbers
$e_n(B_2^d,\nnn{\cdot})$. We will accomplish this using a classical 
result, the dual Sudakov inequality of N.\ Tomczak-Jaegermann
\cite[Lemma 8.3.6]{Tal14}.

\begin{proof}[Proof of Theorem \ref{thm:dimfreerud}]
We use the trivial estimate
$$
	\nnn{v}^2 \le \|v\|_\sim :=
	\sup_{z\in B_2^d}\|v\|_z
$$
for $v\in B_2^d$. This implies, using Remark \ref{rem:propercov}
and the dual Sudakov inequality, that
$$
	e_n(B_2^d,\nnn{\cdot})^2 \lesssim
	e_n(B_2^d,\|\cdot\|_\sim) \lesssim
	2^{-n/2}\mathbf{E}\|g\|_\sim,
$$
where $g$ is a standard Gaussian vector in $\mathbb{R}^d$.
Corollary \ref{cor:supernck} yields
$$
	\mathbf{E}\|X\| \lesssim
	\Bigg\|\sum_{k=1}^m A_k^2\Bigg\|^{1/2} +
	\mathbf{E}\|g\|_\sim.
$$
Now suppose $A_k=x_kx_k^*$ have rank one. Then
\begin{align*}
	\mathbf{E}\|g\|_\sim &=
	\mathbf{E}\Bigg[
	\sup_{z\in B_2^d}\sum_{k=1}^m
	\langle z,x_k\rangle^2 \langle x_k,g\rangle^2
	\Bigg]^{1/2} \\
	&\le
	\Bigg[\sup_{z\in B_2^d}
	\sum_{k=1}^m
	\langle z,x_k\rangle^2\|x_k\|^2\log(k+1)\Bigg]^{1/2}
	\mathbf{E}\bigg[
	\max_{k\le m}
	\frac{|\langle x_k,g\rangle|}{\|x_k\|\sqrt{\log(k+1)}}
	\bigg]
	\\
	&\lesssim
	\Bigg\|
	\sum_{k=1}^m A_k^2\log(k+1)
	\Bigg\|^{1/2},	
\end{align*}
using $A_k^2 = x_kx_k^*\|x_k\|^2$ and that
$\mathbf{E}[\max_{k}|G_k|/\sqrt{\log(k+1)}]\lesssim 1$ when
$G_k$ are (not necessarily independent) standard Gaussian 
variables \cite[Proposition 2.4.16]{Tal14}.
\end{proof}

\subsection{Discussion}
\label{sec:disc}

The aim of this section is to briefly discuss the connection between 
Corollary \ref{cor:supernck} and a conjecture of Lata{\l}a. Let us briefly 
recall this conjecture, which is discussed in detail in \cite{vH16b}.
Let $X$ be a symmetric $d\times d$ matrix whose entries $\{X_{ij}:i\ge j\}$
are independent centered Gaussians with arbitrary variances 
$X_{ij}\sim N(0,b_{ij}^2)$. Lata{\l}a's conjecture states that the 
spectral norm of such a matrix is always of the same order as the maximum 
of the Euclidean norm of its rows,
$$
	\mathbf{E}\|X\| \stackrel{?}{\asymp}
	\mathbf{E}\Bigg[\max_i\sqrt{\sum_j X_{ij}^2}\Bigg].
$$
The lower bound is trivial, as the spectral norm of any matrix is bounded 
below (deterministically) by the maximal Euclidean norm of its rows. It is 
far from obvious, however, why the upper bound should be true.

The independent entry model can be equivalently written as
$$
	X = \sum_{i\ge j} g_{ij}A_{ij},\qquad\quad
	A_{ij} = b_{ij}(e_ie_j^*+e_je_i^*),
$$
where $\{e_i\}$ denotes the standard basis in $\mathbb{R}^d$ and
$\{g_{ij}\}$ are independent standard Gaussian variables. This model is 
therefore a special case of the general model considered in this section.
Unfortunately, the matrices $A_{ij}$ are not positive semidefinite. If the 
conclusion of Corollary \ref{cor:supernck} were to hold nonetheless for 
these matrices, then Lata{\l}a's conjecture would follow readily.
Indeed, arguing precisely as in the proof of Theorem \ref{thm:dimfreerud}, 
we would obtain in this case
\begin{align*}
	\mathbf{E}\|X\| &\stackrel{?}{\lesssim}
	\Bigg\|\sum_{i\ge j} A_{ij}^2\Bigg\|^{1/2} +
	\mathbf{E}\bigg[\sup_{z\in B_2^d}\|g\|_z\bigg]
	\\
	&\lesssim
	\max_i\sqrt{\sum_j b_{ij}^2} +
	\mathbf{E}\Bigg[
	\max_i
	\sqrt{\sum_{j} b_{ij}^2 g_j^2}
	\Bigg]
	\lesssim
	\mathbf{E}\Bigg[\max_i\sqrt{\sum_j X_{ij}^2}\Bigg],
\end{align*}
where the last inequality was established in \cite{vH16b}. In view of 
these observations, it is of significant interest to understand to what 
extent the positive semidefinite assumption made in this section could be 
weakened.

An inspection of the proof of Theorem \ref{thm:gordon} shows that the 
positive semidefinite assumption was used only to ensure that 
$\nnn{\cdot}$ is a norm and that $\|v\|_z\le\nnn{v}\nnn{z}$. All results 
in this section therefore continue to hold verbatim if we were to replace 
$\nnn{\cdot}$ in the statement and proof of Theorem \ref{thm:gordon} and 
Corollary \ref{cor:supernck} by an arbitrary (quasi)norm $\nnn{\cdot}'$ 
such that $\|v\|_z\lesssim \nnn{v}'\nnn{z}'$. This makes it possible, in 
principle, to prove much more general versions of these results. For 
example, the norm
$$
	\nnn{v}' = \Bigg[\sum_{k=1}^m \langle v,|A_k|v\rangle^2
	\Bigg]^{1/4}
$$
satisfies the requisite condition for arbitrary $A_1,\ldots,A_m$, so that 
we obtain a general variant of Theorem \ref{thm:gordon} and
Corollary \ref{cor:supernck} without any assumption on the coefficient 
matrices. However, significant loss may be incurred when we replace
$A_k$ by $|A_k|$. For example, in the 
independent entry model this yields a bound of the form
$$
	\mathbf{E}\|X\| \lesssim
	\max_i\sqrt{\log i}\sqrt{\sum_j b_{ij}^2},
$$
which is far larger than the bound suggested by Lata{\l}a's conjecture.

Other choices of $\nnn{\cdot}'$ are possible in specific situations. For 
example, in the independent entry model, consider the choice
$$
	\nnn{v}' = \Bigg[\sum_{i,j=1}^d v_i^2b_{ij}^2v_j^2
        \Bigg]^{1/4}.
$$
This defines a norm if we assume that the matrix of entry variances 
$(b_{ij}^2)$ is positive semidefinite, in which case it is readily 
verified that $\|v\|_z\lesssim \nnn{v}'\nnn{z}'$. This choice suffices to
establish Lata{\l}a's conjecture under the highly restrictive assumption 
that $(b_{ij}^2)\succeq 0$, recovering a result proved in \cite{vH16b} by 
different means.

In more general situations, it is not clear that it is possible to 
introduce a suitable (quasi)norm $\nnn{\cdot}'$ without incurring 
significant loss, and it is likely that the resolution of Lata{\l}a's 
conjecture will require some additional geometric insight. Nonetheless, 
beside their independent interest, the results of this section provide a 
further step toward better understanding of the multiscale geometry of 
random matrices, and suggest that further development of the methods of 
this paper could yield new insights on various open problems in this area.

\begin{rem}
It is worth noting that even when $A_1,\ldots,A_m$ are positive definite, 
the geometric approach developed here is not necessarily efficient. 
Consider, for example, the trivial case where $m=1$ and $A_1=I$ is the 
identity matrix. Then obviously $\mathbf{E}\|X\|\asymp 1$, but Corollary 
\ref{cor:supernck} gives the terrible bound
$$
	\mathbf{E}\|X\| \lesssim 1 + \sup_{n\ge 0} 
	2^{n/2}e_n(B_2^d,\|\cdot\|_2)^2 \asymp \sqrt{d}.
$$
Thus the geometric principle behind this section cannot fully explain the 
noncommutative Khintchine inequality discussed in Remark \ref{rem:khin}, 
even though it actually improves on this inequality when the coefficient 
matrices have low rank. Discovering the correct geometric explanation of 
the noncommutative Khintchine inequality is closely related to another 
fundamental problem in the generic chaining theory \cite[pp.\ 
50--51]{Tal14} whose resolution may also shed new light on other random 
matrix problems (such as, for example, the problem of obtaining sharp 
bounds in \cite{RV08}).
\end{rem}

\subsection*{Acknowledgments}

The author is grateful to Richard Nickl for hosting a very pleasant visit 
to Cambridge during which some key results of this paper were obtained, 
and to Roman Vershynin and Subhro Ghosh for motivating discussions. This 
work was supported in part by NSF grant CAREER-DMS-1148711 and by the ARO 
through PECASE award W911NF-14-1-0094.


\end{document}